\documentclass[reqno]{amsart}
\usepackage{amsfonts}
\usepackage{amssymb}
\usepackage{amsfonts}
\usepackage{pifont}
\usepackage{amsmath}
\usepackage{amssymb, amsmath}
\usepackage{enumerate}
\usepackage{graphicx}
\usepackage{mathrsfs}
\allowdisplaybreaks[4]
\usepackage{hyperref}

\usepackage{color}

\newcommand{\newcom}{\newcommand}

\newcom{\al}{\alpha}
\newcom{\be}{\beta}
\newcom{\eps}{\epsilon}
\newcom{\veps}{\varepsilon}
\newcom{\ga}{\gamma}
\newcom{\Ga}{\Gamma}
\newcom{\ka}{\kappa}
\newcom{\Lam}{\Lambda}
\newcom{\lam}{\lambda}
\newcom{\Om}{\Omega}
\newcom{\om}{\omega}
\newcom{\Si}{\Sigma}
\newcom{\si}{\sigma}
\newcom{\tht}{\theta}
\newcom{\dtri}{\nabla}
\newcom{\tri}{\triangle}
\newcom{\oo}{\infty}
\newcom{\vphi}{\varphi}
\newcom{\cB}{{\mathcal B}}
\newcom{\cC}{{\mathcal C}}
\newcom{\cD}{{\mathcal D}}
\newcom{\cF}{{\mathcal F}}
\newcom{\cL}{{\mathcal L}}
\newcom{\cM}{{\mathcal M}}
\newcom{\cP}{{\mathcal P}}
\newcom{\cS}{{\mathcal S}}
\newcom{\cQ}{{\mathcal Q}}
\newcom{\cT}{{\mathcal T}}
\newcom{\cY}{{\mathcal Y}}
\newcom{\cZ}{{\mathcal Z}}
\newcom{\R}{\mathbb R}
\newcom{\T}{\mathbb T}
\newcom{\N}{\mathbb N}
\newcom{\Z}{\mathbb Z}
\newcom{\C}{\mathbb C}
\newcom{\E}{\mathbb E}
\newcommand{\De}{\Delta}

\newcom{\f}{\frac}
\newcom{\di}{\displaystyle\int}
\newcom{\ds}{\displaystyle\sum}
\newcom{\dl}{\displaystyle\lim}
\newcom{\ov}{\overline}
\newcom{\sset}{\subset}
\newcom{\wt}{\widetilde}
\newcom{\pa}{\partial}
\newcom{\p}{\partial}
\newcom\na{\nabla}
\newcom{\suml}{\sum\limits}
\newcom{\supl}{\sup\limits}
\newcom{\intl}{\int\limits}
\newcom{\infl}{\inf\limits}
\newcom{\disp}{\displaystyle}
\newcom{\non}{\nonumber}
\newcom{\no}{\noindent}
\newcom{\QED}{$\square$}

\newtheorem{athm}{\bf \t}[section]
\newenvironment{thm} [1] {\def\t{#1}\begin{athm} \bf \rm} {\end{athm}}
\newcom{\bthm}{\begin{thm}}
\newcom{\ethm}{\end{thm}}

\newtheorem{theorem}{Theorem}[section]
\newtheorem{lemma}{Lemma}[section]
\newtheorem{remark}{Remark}[section]

\newcom{\beq}{\begin{equation}}
\newcom{\eeq}{\end{equation}}
\newcom{\ben}{\begin{eqnarray}}
\newcom{\een}{\end{eqnarray}}
\newcom{\beno}{\begin{eqnarray*}}
\newcom{\eeno}{\end{eqnarray*}}
\newcom{\bali}{\begin{aligned}}
\newcom{\eali}{\end{aligned}}

\numberwithin{equation}{section}

\begin{document}

\title[LARGE TIME BEHAVIOR OF THE NON-ISENTROPIC  NAVIER-STOKES-MAXWELL SYSTEM]
{LARGE TIME BEHAVIOR FOR THE NON-ISENTROPIC  NAVIER-STOKES-MAXWELL SYSTEM}

\author{Qingqing Liu}
\address{(QQL)
School of Mathematics, South China University of Technology,
Guangzhou, 510641, P. R. China} \email{maqqliu@scut.edu.cn}

\author{Yifan Su}
\address{(YFS)
The Hubei Key Laboratory of Mathematical Physics, School of
Mathematics and Statistics, Central China Normal University, Wuhan,
430079, P. R. China} \email{eschenbach@foxmail.com}

%\thanks{}

\date{\today}


\begin{abstract}
In this paper, we are concerned with the system of the non-isentropic compressible Navier-Stokes equations coupled with the Maxwell equations through the Lorentz force in three space dimensions. The global existence of solutions
near constant steady states is established, and the time-decay
rates of perturbed solutions are obtained. The proof for existence is due
to the classical energy method, and the investigation of large-time behavior is based on the linearized analysis
of the non-isentropic Navier-Stokes-Poisson equations and the electromagnetic part for the linearized isentropic Navier-Stokes-Maxwell equations.  In the meantime, the time-decay rates obtained by Zhang, Li, and Zhu~[{\it J. Differential Equations, 250(2011), 866-891}] for the linearized non-isentropic Navier-Stokes-Poisson equations are improved.
\end{abstract}

\maketitle

%\tableofcontents

%\newpage
\section{Introduction}

% The 2-fluids Euler-Maxwell system in plasma physics describe dynamics of
%two separate compressible fluids of ions and electrons interacting
%with their self-consistent electromagnetic field. Many famous
%nonlinear dispersive PDE, such as Zakharov¡¯s equation, nonlinear
%Schr\"{o}dinger equations, as well as KdV equations, can be formally
%derived from 2-fluids Euler-Maxwell system under various asymptotic
%limits. In this paper, we consider the bipolar Euler-Maxwell system
%in three dimensional space. The system reads

The Navier-Stokes-Maxwell system is a plasma physical model
describing the motion of charged particles (ions and
electrons) in electromagnetic field \cite{Krall,Pai}. One can refer to Appendix in \cite{Renjun Duan} to
see the derivation of the Navier-Stokes-Maxwell system from the
Vlasov-Maxwell-Boltzmann system. In this paper, we
consider the following one-fluid non-isentropic compressible Navier-Stokes-Maxwell system,
\begin{eqnarray}\label{1.1}
&&\left\{\begin{aligned}
&\partial_t \rho+\nabla\cdot(\rho u)=0,\\
&\rho(\partial_t u+u \cdot \nabla u)+\nabla
P(\rho,\theta)=-\rho\left(E+u\times B\right)+\mu \Delta u+(\mu+\mu')\nabla \nabla\cdot u,\\
& \rho c_{\nu}(\partial_{t}\theta+u\cdot \nabla\theta)+\theta
P_{\theta}(\rho,\theta)\nabla\cdot u=\kappa \Delta
\theta+\frac{\mu}{2}|\nabla u+(\nabla
u)'|^{2}+\mu'(\nabla \cdot u)^{2},\\
&\partial_t E-\nabla\times B=\rho u,\\
&\partial_t B+\nabla \times E=0,\\
&\nabla \cdot E=1-\rho, \ \ \nabla \cdot B=0.
\end{aligned}\right.
\end{eqnarray}\\
Here, $\rho =\rho(t,x)\geq 0$ is the electron density, $u=u(t,x)\in \mathbb{R}^{3}$ is the electron velocity, $\theta=\theta(t,x)$ is the absolute temperature, $E=E(t,x)\in \mathbb{R}^{3}$, $B=B(t,x)\in \mathbb{R}^{3}$ for $t>0,\ x\in \mathbb{R}^{3}$, represent electronic and magnetic fields respectively. The pressure function $P(\cdot)\in \mathbb{R}_{+}$ depending  on $\rho$ and $\theta$ is smooth. Initial data are given as
\begin{equation}\label{1.2}
[\rho,u,\theta, E,B]|_{t=0}=[\rho_{0},u_{0},\theta_{0},E_{0},B_{0}],\ x\in \mathbb{R}^{3},
\end{equation}
with the compatible conditions
\begin{equation}\label{1.3}
\nabla\cdot E_{0}=1-\rho_{0},\ \nabla\cdot B_{0}=0,
\end{equation}
which is crucial in our paper to improve the asymptotic results on Navier-Stokes-Poisson equations in \cite{ZLZ}. The constant viscous coefficients $\mu$ and $\mu'$ satisfy $\mu+\frac{2}{3}\mu'>0$.

The global existence and large time behavior of this model have been studied by Duan \cite{Renjun Duan} in the isentropic case. It inspires a series of researches on this model. For instance, Hou-Yao-Zhu \cite{HYZ} established the global existence and uniqueness of strong solutions to the 1-D compressible Navier-Stokes-Maxwell system with large initial data for the initial boundary value problem. Hong-Hou-Peng-Zhu~\cite{HHPZ} investigated the global existence of spherically symmetric classical solution to the Navier-Stokes-Maxwell system with large initial data and vacuum. Meanwhile, Feng-Peng-Wang \cite{FPW} considered the full compressible Navier-Stokes-Maxwell equations where the temperature equation takes the form of
\begin{equation*}
\theta_{t}+\frac{2}{3}u\cdot \nabla \theta+u\cdot \nabla \theta+(\theta-1)+\frac{1}{3}|u|^{2}=0.
\end{equation*}
They proved the global existence and large-time behavior but without decay rate. Later, Wang-Xu \cite{WX} continued to study the full compressible Navier-Stokes-Maxwell system appeared in \cite{FPW} and obtained the time-decay rate of the global smooth solutions based on a detailed analysis to the Green's function of the linearized system and some elaborate energy estimates.

Another interesting model is Navier-Stokes-Poisson system when the magnetic field is absent. A detailed spectrum analysis of the linear semigroup in terms of the decomposition of wave modes at both lower and higher frequency can be found in \cite{LMZ}, which is important to
capture the convergence rates of solutions to the corresponding linearized equation. Later on, Wang-Wu \cite{WW} obtained the pointwise estimates of solutions to the Navier-Stokes-Poisson system based on the analysis of the Green's function. And in \cite{ZLZ}, Zhang-Li-Zhu generalized the isentropic case to the non-isentropic case with additional efforts in taking care of the temperature equation. See \cite{DL}  \cite{DY} \cite{TWW}  and reference therein for discussion and analysis of the existence results around nontrivial profile, such as stability of steady states with non-flat doping profile \cite{TWW}, stability of rarefaction waves and boundary layer for outflow on the two-fluid case \cite{DY}, stability of rarefaction waves \cite{DL}. See \cite{HLi} for the the global existence of solutions in the Besov space, See
\cite{ZT} for a study in two space dimensions and a survey by Hsiao-Li \cite{HL}.

On the other hand, there is no analysis on the global
existence and the asymptotic behavior for the full compressible Navier-Stokes-Maxwell equations derived from the Vlasov-Maxwell-Boltzmann system \cite{Renjun Duan}. In this paper, we will prove the global existence and time-decay rates of the solutions to the above Navier-Stokes-Maxwell system \eqref{1.1} around constant equilibrium  $[\rho,u,\theta,E,B]=[1,0,1,0,0]$. Our main result can be stated as follows.

\begin{theorem}\label{th1}
Assume that $N\geq 4$ and initial data $U_{0}:=[\rho_{0},u_{0},\theta_{0},E_{0},B_{0}]$ satisfies the compatible condition \eqref{1.3}. There is $ \delta_{0}>0$ such that if
\begin{eqnarray*}
\|[\rho_{0}-1,u_{0},\theta_{0}-1,E_{0},B_{0}]\|_{N} \leq \delta_{0},
\end{eqnarray*}
then, the Cauchy problem \eqref{1.1}-\eqref{1.3} of the non-isentropic Navier-Stokes-Maxwell system admits a unique
global solution  $[\rho(t,x),u(t,x),\theta(t,x),E(t,x),B(t,x)] $ with
\begin{eqnarray*}
\begin{aligned}
&[\rho(t,x)-1,u(t,x),\theta(t,x)-1,E(t,x),B(t,x)]\in
C([0,\infty);H^{N}(\mathbb{R}^{3})),\\
&\rho-1\in L^{2}([0,\infty);H^{N }(\mathbb{R}^{3})),\\
&\nabla u \in L^{2}([0,\infty);H^{N}(\mathbb{R}^{3})),\ \ \ \nabla\theta  \in L^{2}([0,\infty);H^{N}(\mathbb{R}^{3})),\\
&\nabla E\in L^{2}([0,\infty);H^{N-2}(\mathbb{R}^{3})),\ \ \ \nabla^{2} B\in L^{2}([0,\infty);H^{N-3}(\mathbb{R}^{3})),
\end{aligned}
\end{eqnarray*}
Moreover, there is $\delta_{1}>0$ such that if
\begin{eqnarray*}
\|[\rho_{0}-1,u_{0},\theta_{0}-1,E_{0},B_{0}]\|_{L^{1}\cap
H^{N+2}}\leq
\delta_{1},
\end{eqnarray*}
then, the solution $[\rho(t,x),u(t,x),\theta(t,x),E(t,x),B(t,x)] $
satisfies the following decay properties,
\begin{eqnarray}
&&\begin{aligned}
&\|\rho(t)-1\|\lesssim(1+t)^{-1},\\
&\|u(t)\|\lesssim(1+t)^{-\frac{5}{8}},\\
&\|[E(t),\theta(t)-1]\|\lesssim(1+t)^{-\frac{3}{4}}\ln(1+t),\\
&\|B(t)\|\lesssim(1+t)^{-\frac{3}{8}},\\
\end{aligned}
\end{eqnarray}
for any $t\geq 0$.
\end{theorem}

One of the most important characteristics of \eqref{1.1} is  its regularity-loss property (see \cite{Renjun Duan}), meaning that the highest-order spacial derivative of the electromagnetic field is not time-space integrable. Another interesting model in plasma physics which has the same feature is the Euler-Maxwell system. One can refer to \cite{Chen,Duan,DLZ,GM,GIP,Jerome,LZ,UK,USK} for a detailed knowledge on this subject. There has been a general theory developed in \cite{UDK} in terms of the Fourier energy method to study the decay structure of general symmetric hyperbolic systems with partial relaxations of the regularity-loss type. The main feature of time-decay properties for such regularity-loss system is that solutions over the high-frequency domain can still gain the enough time-decay rate by compensating enough regularity of initial data.  On the other hand, the velocity and the temperature themselves are not time-space integrable due to the Navier-Stokes structure of \eqref{1.1}. We must be very careful to deal with the zero-order estimates.

The time decay rates stated in Theorem \ref{th1} depend essentially on the spectral analysis of the linearized system around the constant steady state. In fact, the solution to the linearized homogenous system can be written as the sum of the fluid part and the electromagnetic part in the form of
\begin{eqnarray*}
&& \left[
  \begin{array}{c}
   n(t,x)\\
   u(t,x)\\
   \sigma(t,x)\\
   E(t,x)\\
   B(t,x)\\
 \end{array}\right]
 =
 \left[
  \begin{array}{c}
   n(t,x)\\
   u_{\parallel}(t,x)\\
   \sigma(t,x)\\
   E_{\parallel}(t,x)\\
   0\\
 \end{array}\right]+
\left[
  \begin{array}{c}
   0 \\
   u_{\perp}(t,x)\\
     0\\
   E_{\perp}(t,x)\\
  B(t,x)\\
 \end{array}\right],
\end{eqnarray*}
where the two terminologies, {\it fluid part} and {\it electromagnetic part}, have been used in \cite{Duan,Renjun Duan,DuanLZ}. With the help of the above decomposition, the linearized system \eqref{Linear} of the non-isentropic Navier-Stokes-Maxwell equations can be rewritten as two subsystems, the fluid part are in accordance with the linearized Navier-Stokes-Poisson equations in \cite{ZLZ}, and the electromagnetic part coincides with that for the isentropic Navier-Stokes-Maxwell system in \cite{Renjun Duan}. This decomposition is quite useful in dealing with complex linearized system containing curl, such as  Euler-Maxwell and Navier-Stokes-Maxwell system. We can  refer to \cite{Duan,Renjun Duan,DuanLZ} for the detailed spectrum analysis with the use of a similar decomposition.

Compared with the decay rate $(1+t)^{-\frac{5}{8}}$ of the velocity for the linearized isentropic Navier-Stokes-Maxwell equations in \cite{Renjun Duan}, the velocity of the linearized Navier-Stokes-Poisson system decays at the slower time-decay rate $(1+t)^{-\frac{1}{4}}$, which is not enough to generalize the isentropic results \cite{Renjun Duan} to our non-isentropic case. A careful observation of the compatible conditions
$\nabla\cdot E_{0}=1-\rho_{0}$ implies that the singularities at lower frequency in \cite{ZLZ} can be eliminated and thus the time-decay rate $(1+t)^{-\frac{1}{4}}$ of momentum in \cite{ZLZ} can be improved to $(1+t)^{-\frac{3}{4}}$, see more details in Section \ref{sec4}. We also mention the works, Duan-Liu-Ukai-Yang in \cite{DLUY}, Duan-Ukai-Yang-Zhao in
\cite{RSYZ} for the study of
optimal convergence rates of the isentropic and non-isentropic compressible Navier-Stokes
equations with potential forces. The Navier-Stokes-Maxwell equations can be viewed as the Navier-Stokes equations with Lorentz force $E+u\times B$, where the electromagnetic field satisfies Maxwell equations. In this sense, our result is also the  generalization of the result in \cite{RSYZ} for the non-isentropic Navier-Stokes equations with external force.

\medskip
\noindent{\it Notations.} Let us introduce some notations for the use throughout this paper.
$C$ denotes some positive (generally large) constant and $ \lambda$
denotes some positive (generally small) constant, where both $C$ and
$ \lambda$ may take different values in different places. For two
quantities $a$ and $b$, $a\sim b$ means $\lambda a \leq  b \leq
\frac{1}{\lambda} a $ for a generic constant $0<\lambda<1$. For any
integer$m\geq 0$, we use $H^{m}$, $\dot{H}^{m}$ to denote the usual
Sobolev space $H^{m}(\mathbb{R}^{3})$ and the corresponding
$m$-order homogeneous Sobolev space, respectively. Set $L^{2}=H^{m}$
when $m = 0$. For simplicity, the norm of $ H^{m}$ is denoted by
$\|\cdot\|_{m} $ with $\|\cdot \|=\|\cdot\|_{0}$. We use $
\langle\cdot, \cdot \rangle$ to denote the inner product over the
Hilbert space $ L^{2}(\mathbb{R}^{3})$, i.e.
\begin{eqnarray*}
\langle f,g \rangle=\int_{\mathbb{R}^{3}} f(x)g(x)dx,\ \ \ \  f =
f(x),\ \  g = g(x)\in L^2(\mathbb{R}^{3}).
\end{eqnarray*}
 For a multi-index $\alpha =
[\alpha_1, \alpha_2, \alpha_3]$, we denote $\partial^{\alpha} =
\partial^{\alpha_{1}}_ {x_1}\partial^{\alpha_{2}}_ {x_2} \partial^{\alpha_{3}}_ {x_3} $.
The length of $ \alpha$ is $|\alpha| = \alpha_1 + \alpha_2 +
\alpha_3$. For simplicity, we also
set $\partial_{j}=\partial_{x_{j}}$ for $j = 1, 2, 3$.

%\vspace{6mm}

The rest of the paper is organized as follows. In Section \ref{sec2}, we reformulate the Cauchy problem \eqref{1.1}-\eqref{1.3} around the constant steady state. In Section \ref{sec4}, we decompose the linearized equations \eqref{Linear} into two decoupled subsystems which have been studied before in \cite{Renjun Duan,ZLZ} and present the time-decay properties for the linearized equations \eqref{Linear}. In Section \ref{sec5}, we first prove the global existence of solutions by the energy method, and then show the time decay rate of solutions around the constant states.
\vskip 2mm

\section{Reformulation of the problem}\label{sec2}
In this section, we will reformulate the Cauchy problem \eqref{1.1}-\eqref{1.3}. For that, let us set $n=\rho-1$, $\sigma=\theta-1$. Then
$U:=[n,u,\sigma,E,B]$ satisfies
\begin{equation}\label{2.5}
\left\{
  \begin{aligned}
  &\partial_t n+\nabla\cdot u=h_{1},\\
  &\partial_t u+P_{\rho}(1,1)\nabla n+P_{\theta}(1,1)\nabla \sigma+E-\mu\Delta u-(\mu+\mu')\nabla\nabla \cdot u=h_{2},\\
  &\partial_t \sigma +\frac{P_{\theta}(1,1)}{c_{\nu}}\nabla \cdot u-\frac{\kappa}{c_{\nu}}\Delta \sigma=h_{3},\\
  &\partial_t E-\nabla\times B-u=h_{4},\\
  &\partial_t B+\nabla \times E=0,\\
  &\nabla \cdot E=-n, \ \  \nabla\cdot B=0.
\end{aligned}\right.
\end{equation}
Initial data are given by
\begin{eqnarray}\label{NI}
[n,u,\sigma,E,B]|_{t=0}=[\rho_{0}-1,u_{0},\theta_{0}-1,E_{0},B_0],%\ \ \ x\in\mathbb{R}^{3},
\end{eqnarray}
with the compatibility condition
\begin{eqnarray}\label{NC}
\nabla \cdot E_0=-n_{0}, \ \
\nabla \cdot B_0=0.% \ \ \ x\in\mathbb{R}^{3}.
\end{eqnarray}
Here the inhomogeneous source terms are
\begin{equation}\label{sec5.ggg}
\arraycolsep=1.5pt \left\{
 \begin{aligned}
 &h_{1}=-\nabla\cdot(n u)=-\nabla\cdot h_{4},\\
 &\begin{array}{ll}
 h_{2}=&\displaystyle- u \cdot \nabla
u-\left[\frac{P_{\rho}(n+1,\sigma+1)}{n+1}-\frac{P_{\rho}(1,1)}{1}\right]\nabla
n
-\left[\frac{P_{\theta}(n+1,\sigma+1)}{n+1}-\frac{P_{\theta}(1,1)}{1}\right]\nabla
\sigma\\[3mm]
&\displaystyle-u\times B +\left[\frac{\mu}{n+1}-\frac{\mu}{1}\right]\Delta u+\left[\frac{\mu+\mu'}{n+1}-\frac{\mu+\mu'}{1}\right]\nabla \nabla \cdot u,\\
\end{array}\\[3mm]
 &\begin{array}{ll}
 h_{3}=&\displaystyle-u \cdot \nabla
\sigma-\left[\frac{(\sigma+1)P_{\theta}(n+1,\sigma+1)}{c_{\nu}(n+1)}-\frac{P_{\theta}(1,1)}{c_{\nu}}\right]\nabla
\cdot u
+\left[\frac{\kappa}{c_{\nu}(n+1)}-\frac{\kappa}{c_{\nu}}\right]\Delta
\sigma\\[3mm]
&\displaystyle+\frac{\mu}{2c_{\nu}(n+1)}|\nabla u+(\nabla
u)'|^{2}+\frac{\mu'}{c_{\nu}(n+1)}(\nabla \cdot u)^{2},\\
\end{array}\\[3mm]
 &h_{4}=nu.
\end{aligned}\right.
\end{equation}

For brevity of presentation we still use $U=[n,u,\sigma,E,B]$ to
denote the solution to  the linearized homogeneous system
\begin{equation}\label{Linear}
\left\{
  \begin{aligned}
  &\partial_t n+\nabla\cdot u=0,\\
  &\partial_t u+P_{\rho}(1,1)\nabla n+P_{\theta}(1,1)\nabla \sigma+E-\mu\Delta u-(\mu+\mu')\nabla\nabla \cdot u=0,\\
  &\partial_t \sigma +\frac{P_{\theta}(1,1)}{c_{\nu}}\nabla \cdot u-\frac{\kappa}{c_{\nu}}\Delta \sigma=0,\\
  &\partial_t E-\nabla\times B-u=0,\\
  &\partial_t B+\nabla \times E=0,\\
  &\nabla \cdot E=-n, \ \  \nabla\cdot B=0.
  \end{aligned}\right.
\end{equation}
with the given initial data
\begin{eqnarray}\label{NLI}
[n,u,\sigma,E,B]|_{t=0}=[\rho_{0}-1,u_{0},\theta_{0}-1,E_{0},B_0],%\ \ \ x\in\mathbb{R}^{3},
\end{eqnarray}
satisfying the compatibility  condition
\begin{eqnarray}\label{LNC}
\nabla \cdot E_0=-n_{0}, \ \
\nabla \cdot B_0=0.%, \ \ \ x\in\mathbb{R}^{3}.
\end{eqnarray}

\section{Time-decay property of the linearized homogeneous system} \label{sec4}

In order to study the more accurate large-time asymptotic profile,
we need to carry out the spectral analysis of the linearized system \eqref{Linear}.
As in \cite{Duan}, the linearized system \eqref{Linear} can be
written as two decoupled subsystems which govern the time evolution
of $n$, $\nabla\cdot u$, $\sigma$, $\nabla \cdot E$ and
$\nabla\times u$, $\nabla\times E$ and $\nabla \times B$
respectively. We decompose the solution to
\eqref{Linear}-\eqref{LNC} into two parts in the form of
\begin{eqnarray}\label{decompo}
&& \left[
  \begin{array}{c}
   n(t,x)\\
   u(t,x)\\
   \sigma(t,x)\\
   E(t,x)\\
   B(t,x)\\
 \end{array}\right]
 =
 \left[
  \begin{array}{c}
   n(t,x)\\
   u_{\parallel}(t,x)\\
   \sigma(t,x)\\
   E_{\parallel}(t,x)\\
   0\\
 \end{array}\right]+
\left[
  \begin{array}{c}
   0 \\
   u_{\perp}(t,x)\\
     0\\
   E_{\perp}(t,x)\\
  B(t,x)\\
 \end{array}\right],
\end{eqnarray}
where $u_{\parallel}$, $u_{\perp}$ are defined by
\begin{equation*}
u_{\parallel}=-(-\Delta)^{-1}\nabla\nabla\cdot u,\ \ \
u_{\perp}=(-\Delta)^{-1}\nabla\times(\nabla\times u),
\end{equation*}
and likewise for $E_{\parallel}$, $E_{\perp}$. For brevity, the
first part on the right of \eqref{decompo} is called the fluid part
and the second part is called the electromagnetic part, and we also
denote
\begin{equation*}
U_\parallel  =  [n,u_{\parallel},\sigma], \quad U_\perp =
[u_{\perp},E_{\perp},B].
\end{equation*}
Notice that to the end, $E_{\parallel}$ is always given by
$$
E_{\parallel}=(-\De)^{-1}\nabla n.
$$
The fluid part $U_\parallel$ satisfies
\begin{equation}\label{Fluid1}
\left\{
  \begin{aligned}
  &\partial_t n +\nabla\cdot u_{\parallel}=0,\\
  &\partial_t  u_{\parallel}+\alpha_{1}\nabla n+\alpha_{2}\nabla
  \sigma
  +E_{\parallel}-\mu \Delta u_{\parallel}-(\mu+\mu')\nabla \nabla \cdot u_{\parallel}=0,\\
  &\partial_t \sigma +\alpha_{3}\nabla\cdot u_{\parallel}-\bar{\kappa}\Delta \sigma=0.
\end{aligned}\right.
\end{equation}
Initial data are given by
\begin{equation}\label{Fluid2I}
[n,\ u_{\parallel},\sigma]|_{t=0}=[n_{0},\
u_{0,\parallel},\sigma_{0}].
\end{equation}
Here, we have denoted $\alpha_{1}=P_{\rho}(1,1),\
\alpha_{2}=P_{\theta}(1,1),\
\alpha_{3}=\frac{P_{\theta}(1,1)}{c_{\nu}},\
\bar{\kappa}=\frac{\kappa}{c_{\nu}}$.

Notice that the fluid part \eqref{Fluid1} of the linearized equation for non-isentropic Navier-Stokes-Maxwell are in accordance with the linearized Navier-Stokes-Poisson equations in \cite{ZLZ}. Here  and in the sequel, we still use the same notations for the
convenience of comparison with \cite{ZLZ}. For the clear reference, we list some properties of the above linearized equations \eqref{Fluid1} as follows.

After taking the Fourier transformation in $x$ for \eqref{Fluid1},
replacing $\hat{E}_{\parallel}$ by $\frac{i\xi}{|\xi|^{2}}\hat{n}$,
the fluid part
$\hat{U}_{\parallel}=[\hat{n},\hat{u}_{\parallel},\hat{\sigma}]$
satisfies the following system of 1st-order ODEs
\begin{equation}\label{FluidF}
\left\{
  \begin{aligned}
  &\partial_t \hat{n}+i\xi\cdot \hat{u}_{\parallel}=0,\\
  &\partial_t \hat{u}_{\parallel}+\alpha_{1}i\xi\hat{n}+\alpha_{2}i\xi\hat{\sigma}
  +\frac{i\xi}{|\xi|^{2}}\hat{n}+\mu|\xi|^{2}\hat{u}_{\parallel}+(\mu+\mu')\xi\xi\cdot\hat{u}_{\parallel}=0,\\
  &\partial_t \hat{\sigma}+\alpha_{3}i\xi\cdot \hat{u}_{\parallel}+\bar{\kappa}|\xi|^{2}\hat{\sigma}=0.\\
\end{aligned}\right.
\end{equation}
Initial data are given as
\begin{equation}\label{FluidFI}
\hat{U}_{\parallel}(t,\xi)|_{t=0}=\hat{U}_{\parallel
0}(\xi)=:[\hat{n}_{0},\ \tilde{\xi}\tilde{\xi}\cdot \hat{u}_{0},\
\hat{\sigma}_{0}].
\end{equation}
Here we set $\tilde{\xi}=\xi/|\xi|$ for $|\xi|\neq 0$. Then the solution to \eqref{FluidF}, \eqref{FluidFI} can be written
as
\begin{equation*}
\hat{U}_{\parallel}(t,\xi)^T={e^{A(i\xi)t}}\hat{U}_{\parallel
0}(\xi)^{T}:=(a_{ij}(\xi,t))_{5\times5}\hat{U}_{\parallel
0}(\xi)^{T},
\end{equation*}
with the matrix $A(i\xi)$ defined by
\begin{equation}\label{matrix}
A(i\xi)=:\left(\begin{array} {ccc}
0 \ \ & -i\xi^{t} &\ \ \ 0\\
-i\xi(\alpha_{1}+|\xi|^{-2})\ \   & -\mu|\xi|^{2}I_{3\times 3}-(\mu+\mu')\xi\otimes \xi  &\ \ \  -i\xi \alpha_{2}\\
0\ \  & -i\xi^{t}\alpha_{3} & -\bar{\kappa}|\xi|^{2}
\end{array} \right)_{5\times 5}.
\end{equation}

 The Fourier transformation
$(\hat{n},\hat{u}_{\parallel},\hat{\sigma},\hat{E}_{\parallel})$ of
solution $(n,u_{\parallel},\sigma,E_{\parallel})$ can be expressed
as
\begin{equation*}
  \begin{aligned}
  &\hat{n}(\xi,t)=N(t)\hat{U}_{0,\parallel}(\xi)=(N_{1}(t)+N_{2}(t))\hat{U}_{0,\parallel}(\xi),\\
  & \hat{u}_{\parallel}(\xi,t)=M(t)\hat{U}_{0,\parallel}(\xi)=(M_{1}(t)+M_{2}(t))\hat{U}_{0,\parallel}(\xi),\\
  &\hat{\sigma}(\xi,t)=Q(t)\hat{U}_{0,\parallel}(\xi)=(Q_{1}(t)+Q_{2}(t))\hat{U}_{0,\parallel}(\xi),\\
  &\hat{E}_{\parallel}(\xi,t)=\frac{i\xi}{|\xi|^{2}}\hat{n}(\xi,t)=\frac{i\xi}{|\xi|^{2}}(N_{1}(t)+N_{2}(t))\hat{U}_{0,\parallel}(\xi),
\end{aligned}
\end{equation*}
where the matrices $N_{i}(t),\ M_{i}(t),\ Q_{i}(t)$ $(i=1,2)$ are defined as
\begin{equation*}
  \begin{aligned}
  & N_{1}(t)=N_{1}(\xi,t)=(a_{11}(\xi,t),0)_{1\times 5},\ \  N_{2}(t)=N_{2}(\xi,t)=(0,a_{1l}(\xi,t))_{1\times 5},\ \ l=2,3,4,5,\\
  & M_{1}(t)=M_{1}(\xi,t)=(a_{k1}(\xi,t),0)_{3\times 5},\ \  M_{2}(t)=M_{2}(\xi,t)=(0,a_{kl}(\xi,t))_{3\times 5},\ \ k=2,3,4,\ l=2,3,4,5,\\
  & Q_{1}(t)=Q_{1}(\xi,t)=(a_{51}(\xi,t),0,a_{55}(\xi,t))_{1\times 5},\ \  Q_{2}(t)=Q_{2}(\xi,t)=(0,a_{5l}(\xi,t),0)_{1\times 5},\ \ l=2,3,4.\\
\end{aligned}
\end{equation*}
The following time-frequency pointwise estimates can be found in \cite{ZLZ}, which is crucial to the $L^{2}$ time-decay rates of solutions to linearized equations \eqref{Fluid1}.
\begin{lemma} [\textbf{Lemma 2.4} in \cite{ZLZ}]
For $0<|\xi|\leq r_{1}$, we have
\begin{equation*}
  \begin{aligned}
  & |||N_{1}(t)|||\leq Ce^{-\lambda|\xi|^{2} t},\ \  |||N_{2}(t)|||\leq C|\xi|e^{-\lambda|\xi|^{2} t},\\
  & |||M_{1}(t)|||\leq C\frac{1}{|\xi|}e^{-\lambda|\xi|^{2}  t},\ \  |||M_{2}(t)|||\leq Ce^{-\lambda|\xi|^{2} t},\\
  & |||Q_{1}(t)|||\leq Ce^{-\lambda |\xi|^{2} t},\ \  |||Q_{2}(t)|||\leq C|\xi|e^{-\lambda|\xi|^{2} t},
\end{aligned}
\end{equation*}
where $C > 0,\ \lambda>0$ are  positive constants.
\end{lemma}

\begin{lemma} [\textbf{Lemma 2.5} in \cite{ZLZ}]
For $|\xi|>r_{1}$, we have
\begin{equation*}
 |||N(t),M(t),Q(t)|||\leq C(1+t)^{3}e^{-\lambda t},
\end{equation*}
where $C > 0,\ \lambda>0$ are  positive constants.
\end{lemma}

\begin{remark}It is obvious that the time-pointwise estimate  of $|||M_{1}(t)|||$ over low frequency leads to the slow
time-decay rate $(1+t)^{-\frac{1}{4}}$ in $L^2$-norm. However, we may improve the time-decay properties in \cite{ZLZ} below. The reasoning is that we can easily eliminate the singularity $\frac{1}{|\xi|}$ in $M_{1}(t)$ through  combining  the compatible condition $i\xi \cdot \hat{E}_{0}(\xi)=-\hat{n}_{0}(\xi)$ and  the fact that $M_{1}(t)$ is only the coefficient of $\hat{n}_{0}(\xi)$. Similarly the time-decay estimate for $\|\partial^{k} n(t)\|$ can also be improved.
\end{remark}
\begin{lemma}[Improved version of \textbf{Lemma 2.7} in \cite{ZLZ}]\label{lemma.decay1}
The solution $[n,u_{\parallel},\sigma,E_{\parallel}]$ to the Cauchy
problem $\eqref{Fluid1}$-$\eqref{Fluid2I}$ satisfies the following
time-decay property:
\begin{eqnarray}\label{col.decay1}
 && \left\{
 \begin{aligned}
 & \|\partial^{k} n(t)\|\leq C (1+t)^{-\left(\frac{5}{4}+\frac{k}{2}\right)}\left(\|[E_{0},u_{0,\parallel},\sigma_{0}]\|_{L^{1}}
 +\|\partial^{k}[n_{0},u_{0,\parallel},\sigma_{0}]\|\right),\\
 & \|\partial^{k} u_{\parallel}(t)\|\leq C (1+t)^{-\left(\frac{3}{4}+\frac{k}{2}\right)}\left(\|[E_{0},u_{0,\parallel},\sigma_{0}]\|_{L^{1}}
 +\|\partial^{k}[n_{0},u_{0,\parallel},\sigma_{0}]\|\right),\\
 & \|\partial^{k} \sigma(t)\|\leq C (1+t)^{-\left(\frac{3}{4}+\frac{k}{2}\right)}\left(\|[n_{0},u_{0,\parallel},\sigma_{0}]\|_{L^{1}}
 +\|\partial^{k}[n_{0},u_{0,\parallel},\sigma_{0}]\|\right),\\
 & \|\partial^{k} E_{\parallel}(t)\|\leq C (1+t)^{-\left(\frac{3}{4}+\frac{k}{2}\right)}\left(\|[E_{0},u_{0,\parallel},\sigma_{0}]\|_{L^{1}}
 +\|\partial^{k}[n_{0},u_{0,\parallel},\sigma_{0}]\|\right),\\
\end{aligned}\right.
\end{eqnarray}
where $k=0,1,2,\cdots$ and $C>0$ is a positive constant.
\end{lemma}

The electromagnetic part $[u_{\perp}(t,x),E_{\perp}(t,x),B(t,x)]$ satisfies the following equations:
\begin{equation}\label{magnetic1}
 \left\{
  \begin{aligned}
  &\partial_t u_{\perp} +E_\perp-\mu \Delta u_{\perp}=0,\\
  &\partial_t  E_{\perp}-\nabla\times B - u_{\perp}=0,\\
  &\partial_t B +\nabla\times E_{\perp}=0,
\end{aligned}\right.
\end{equation}
with initial data
\begin{equation}\label{magnetic2}
[u_{\perp},E_{\perp},B]=[u_{0,\perp},E_{0,\perp},B_{0}].
\end{equation}

The above linearized system coincides with  the electromagnetic part for the linearized isentropic Navier-Stokes-Maxwell system in \cite{Renjun Duan}. By taking Fourier transformation to the above linearized system \eqref{magnetic1}-\eqref{magnetic2}, one has

\begin{equation}\label{magnetic1F}
 \left\{
  \begin{aligned}
  &\partial_t \hat{u}_{\perp} +\hat{E}_\perp+\mu |\xi|^{2} \hat{u}_{\perp}=0,\\
  &\partial_t  \hat{E}_{\perp}-i\xi\times \hat{B} - \hat{u}_{\perp}=0,\\
  &\partial_t \hat{B} +i\xi\times \hat{E}_{\perp}=0,
\end{aligned}\right.
\end{equation}
with initial data
\begin{equation}\label{magnetic2F}
[\hat{u}_{\perp},\hat{E}_{\perp},\hat{B}]=[\hat{u}_{0,\perp},\hat{E}_{0,\perp},\hat{B}_{0}].
\end{equation}
As in \cite{Renjun Duan}, the solution $[\hat{u}_{\perp},\hat{E}_{\perp},\hat{B}]$ is expressed by
$$
[\hat{u}_{\perp},\hat{E}_{\perp},\hat{B}]^{T}=\mathcal{M} _{9\times 9}(t,\xi)[\hat{u}_{0,\perp},\hat{E}_{0,\perp},\hat{B}_{0}]^{T},
$$
with the following estimates on the matrix $\mathcal{M}_{9\times 9}(t,\xi)$.
\begin{lemma}[\textbf{Proposition 5.2} in \cite{Renjun Duan}] The elements of $\mathcal{M}$ have the following upper bounds for pointwise $ t\geq 0$, and $\xi\in \mathbb{R}^{3}$.

Over $D_{0}=\{\xi||\xi|^{2}\leq \epsilon\}$,
\begin{equation*}
|\mathcal{M}|\lesssim \left(
  \begin{array}{ccc}
 |\xi|^{2} & |\xi|^{4} & |\xi|\\
  |\xi|^{4} & |\xi|^{6} & |\xi|^{3}\\
  |\xi|  & |\xi|^{3} & 1\\
\end{array}\right)e^{-O(1)|\xi|^{4}t}+
\left(
  \begin{array}{ccc}
 1 & 1  & |\xi|\\
  1 & 1  & |\xi|\\
  |\xi|  & |\xi| & |\xi|^{2}\\
\end{array}\right)e^{-O(1)|\xi|^{2}t}.
\end{equation*}

Over $D_{\infty}=\{\xi||\xi|^{2}\geq L\}$,
\begin{equation*}
|\mathcal{M}|\lesssim \left(
  \begin{array}{ccc}
 1 & |\xi|^{-2} & |\xi|^{-3}\\
  |\xi|^{-2} & |\xi|^{-4} & |\xi|^{-5}\\
  |\xi|^{-3}  & |\xi|^{-5} & |\xi|^{-6}\\
\end{array}\right)e^{-O(1)|\xi|^{2}t}+
\left(
  \begin{array}{ccc}
 |\xi|^{-4} & |\xi|^{-2} & |\xi|^{-2}\\
 |\xi|^{-2} & 1 & 1\\
 |\xi|^{-2} & 1 & 1\\
\end{array}\right)e^{-O(1)|\xi|^{-2}t}.
\end{equation*}

Over $D_{1}=\{\xi|\epsilon\leq|\xi|^{2}\leq L\}$,
\begin{equation*}
|\mathcal{M}|\lesssim \left(
  \begin{array}{ccc}
 1 & 1 & 1\\
  1 & 1 & 1\\
  1  & 1 & 1\\
\end{array}\right)e^{-O(1)t}.
\end{equation*}
\end{lemma}

As in \cite{Duan,Renjun Duan}, it is a  standard procedure  to derive the following time-decay property of solutions to the  electromagnetic part $[u_{\perp},E_{\perp},B]$, one can refer to  Theorem 5.2 in \cite{Renjun Duan}.
\begin{lemma}\label{lemma.decay2}
The solution $[u_{\perp},E_{\perp},B]$ to the Cauthy problem $\eqref{magnetic1}$-$\eqref{magnetic2}$
satisfies the following time-decay property:
\begin{eqnarray}\label{Co.5.2}
 && \left\{
 \begin{aligned}
 &\begin{array}{ll}
 \|u_{\perp}\| \lesssim  (1+t)^{-\frac{3}{4}}\|[u_{0},E_{0}]\|_{L^{1}\cap L^{2}}+(1+t)^{-\frac{5}{8}}\|B_{0}\|_{L^{1}\cap L^{2}},\\
 \end{array}\\[2mm]
 &\begin{array}{ll}
 \|E_{\perp}\| \lesssim (1+t)^{-\frac{3}{4}}\|u_{0}\|_{L^{1}\cap L^{2}}+(1+t)^{-\frac{3}{4}}(\|E_{0}
 &\displaystyle\|_{L^{1}\cap L^{2}}+\|\nabla^{2}E_{0}\|)\\[2mm]
 &\displaystyle+(1+t)^{-\frac{9}{8}}(\|B_{0}\|_{L^{1}\cap L^{2}}+\|\nabla^{3}B_{0}\|),\\[2mm]
 \end{array}\\
 & \begin{array}{ll}
 \|B\| \lesssim (1+t)^{-\frac{5}{8}}\|u_{0}\|_{L^{1}\cap L^{2}}+(1+t)^{-\frac{9}{8}}(\|E_{0}&\displaystyle\|_{L^{1}\cap L^{2}}+\|\nabla^{3}E_{0}\|)\\[2mm]
 &\displaystyle +(1+t)^{-\frac{3}{8}}(\|B_{0}\|_{L^{1}\cap L^{2}}+\|\nabla B_{0}\|),\\[2mm]
 \end{array}\\
 &\begin{array}{ll}
 \|\nabla B\| \lesssim (1+t)^{-\frac{7}{8}}(\|u_{0}\|_{L^{1}\cap L^{2}}+\|\nabla u_0\|_{L^{2}})+(1+t)^{-\frac{11}{8}}(\|E_{0}&\displaystyle\|_{L^{1}\cap L^{2}}+\|\nabla^{4}E_{0}\|_{L^{2}})\\[2mm]
 &\displaystyle+(1+t)^{-\frac{5}{8}}(\|B_{0}\|_{L^{1}\cap L^{2}}+\|\nabla^{3}B_{0}\|)
 \end{array}\\
\end{aligned}\right.
\end{eqnarray}
for any $t\geq 0$.
\end{lemma}

\medskip
The time decay rate of the solution to the original linearized equation  \eqref{Linear}-\eqref{NLI} can be  immediately obtained  by combining Lemma \ref{lemma.decay1} and Lemma \ref{lemma.decay2}.

\begin{lemma}\label{lemma.decay3}
The solution $[n,u,\sigma,E,B]$ to the Cauchy problem \eqref{Linear}-\eqref{NLI} have the following time-decay property:
\begin{eqnarray}
 &&\left\{
\begin{aligned}
&\begin{array}{ll}
\|n\|\lesssim (1+t)^{-\frac{5}{4}}\|[n_{0},u_{0},\sigma_{0}]\|_{L^{1}
\cap L^{2}}+(1+t)^{-\frac{5}{4}}\|E_{0}\|_{L^{1}},\\[2mm]
\end{array}\\
&\begin{array}{ll}
\|u\|\lesssim
(1+t)^{-\frac{3}{4}}\|[n_{0},u_{0},\sigma_{0},E_{0}]\|
_{L^{1}\cap L^{2}}+(1+t)^{-\frac{5}{8}}\|B_{0}\|_{L^{1}\cap L^{2}},\\[2mm]
\end{array}\\
&\begin{array}{ll}
\| \sigma \| \lesssim (1+t)^{-\frac{3}{4}}
\|[n_{0},u_{0},\sigma_{0}]\|_{L^{1}\cap L^{2}},\\[2mm]
\end{array}\\
&\begin{array}{ll}
\|E\|\lesssim (1+t)^{-\frac{3}{4}}\|[n_{0},u_{0},\sigma_{0},E_{0}]\|_{L^{1}\cap L^{2}}
+(1+&\displaystyle t)^{-\frac{3}{4}}
\|\nabla^{2}E_{0}\|\\[2mm]
&+(1+t)^{-\frac{9}{8}}(\|B_{0}\|_{L^{1}\cap L^{2}} +\|\nabla^{3}B_{0}\|),\\[2mm]
\end{array}\\
&\begin{array}{ll}
\|B\|\lesssim (1+t)^{-\frac{5}{8}}\|u_{0}\|_{L^{1}\cap L^{2}}
+(1+t)^{-\frac{9}{8}}&\displaystyle(\|E_{0}\|_{L^{1}\cap L^{2}} +\|\nabla^{3}E_{0}\|)\\[2mm]
&+(1+t)^{-\frac{3}{8}}(\|B_{0}\|_{L^{1}\cap L^{2}} +\|\nabla B_{0}\|)
\end{array}\\
\end{aligned}\right.
\end{eqnarray}
for any $t\geq 0$.
\end{lemma}

\section{Asymptotic behavior of the nonlinear system}\label{sec5}

\subsection{Global existence}
In this section, we will establish the global existence of solution to the compressible non-isentropic Navier-Stokes-Maxwell system~\eqref{2.5}-\eqref{NC}.  For later use and clear reference, the following sobolev inequality about the $L^p$ estimate on any two
product terms with the sum of the order of their derivatives equal
to a given integer is listed as follows \cite{DRZ,Taylor}.
\begin{lemma}\label{s5.le1}
Let $n\geq 1$. Let
$\alpha^1=\left(\alpha_1^1,\cdots,\alpha_n^1\right)$ and
$\alpha^2=\left(\alpha_1^2,\cdots,\alpha_n^2\right)$ be two
multi-indices with $\left|\alpha^1\right|=k_1,
\left|\alpha^2\right|=k_2$ and set $k=k_1+k_2$. Let $1\leq
p,q,r\leq\infty$ with $1/p=1/q+1/r$. Then, for $u_j: \R^n\to \R$
$(j=1,2)$, one has
\begin{equation}\label{s5.eq1}
\left\|\partial^{\alpha^1}u_1\partial^{\alpha^2}u_2\right\|_{L^p\left(\mathbb{R}^n\right)}
\leq C\left(\left\|u_1\right\|_{L^q\left(\mathbb{R}^n\right)}
\left\|\nabla^ku_2\right\|_{L^r\left(\mathbb{R}^n\right)}
+\left\|u_2\right\|_{L^q\left(\mathbb{R}^n\right)}
\left\|\nabla^ku_1\right\|_{L^r\left(\mathbb{R}^n\right)}\right)
\end{equation}
for a constant $C$ independent of $u_1$ and $u_2$.
\end{lemma}

To the end, we  assume the integer $N \geq 4$. For
$U=[n,u,\sigma,E,B]$, we define the full instant energy functional
$\mathcal {E}_{N}(U(t))$ and  the high-order instant
%energy functional $\mathcal {E}_{N}^{h}(U(t))$, and the dissipation
%rates $\mathcal {D}_{N}(U(t))$, $\mathcal {D}_{N}^{h}(U(t))$
by
\begin{equation}\label{de.E}
\arraycolsep=1.5pt
\begin{array}{rl}
\mathcal {E}_{N}(U(t))=&\displaystyle
\|[n,u,\sigma,E,B]\|_{N}^{2}+\kappa_{1}\sum_{|\alpha|\leq
N-1}\langle\partial^{\alpha}u,\partial^{\alpha}\nabla n
\rangle\\[3mm]
&\displaystyle+\kappa_{1}\sum_{|\alpha|\leq N-2}\langle
\partial^{\alpha}\nabla\times u,
\partial^{\alpha}\nabla\times E\rangle-\kappa_{1}\kappa_{2}\sum_{1\leq |\alpha|\leq
N-2}\left\langle
\partial^{\alpha}E,\partial^{\alpha}\nabla \times B\right\rangle,
\end{array}
\end{equation}

and
\begin{equation}\label{de.Eh}
\arraycolsep=1.5pt
\begin{array}{rl}
\mathcal {E}^{h}_{N}(U(t))=&\displaystyle
\|\nabla[n,u,\sigma,E,B]\|_{N-1}^{2}+\kappa_{1}\sum_{1\leq
|\alpha|\leq N-1}\langle\partial^{\alpha}u,\partial^{\alpha}\nabla n
\rangle\\[3mm]
&\displaystyle+\kappa_{1}\sum_{1\leq |\alpha|\leq N-2}\langle
\partial^{\alpha}\nabla\times u,
\partial^{\alpha}\nabla\times E\rangle-\kappa_{1}\kappa_{2}\sum_{2\leq |\alpha|\leq
N-2}\left\langle
\partial^{\alpha}E,\partial^{\alpha}\nabla \times B\right\rangle,
\end{array}
\end{equation}
respectively, where $0<\kappa_{2},\ \kappa_{1}\ll 1$ are constants
to be properly chosen later on. Notice that since all constants
$\kappa_i$ $(i=1,2)$ are small enough, one has
\begin{equation*}
    \mathcal {E}_{N}(U(t))\sim
\|[n,u,\sigma,E,B] \|_{N}^{2},\quad \mathcal {E}_{N}^{h}(U(t))\sim
\|\nabla [n,u,\sigma,E,B] \|_{N-1}^{2}.
\end{equation*}
We further define the corresponding dissipation rates $\mathcal
{D}_{N}(U(t))$, $\mathcal {D}_{N}^{h}(U(t))$ by
\begin{equation}\label{de.D}
\mathcal {D}_{N}(U(t))=\displaystyle \|\nabla
u\|_{N}^{2}+\|n\|_{N}^{2}+\|\nabla \sigma\|_{N}^2 +\|\nabla
E\|_{N-2}^{2}+\|\nabla^{2}B\|_{N-3}^{2},
\end{equation}
and
\begin{equation}\label{de.Dh}
\mathcal {D}_{N}^{h}(U(t))=\displaystyle \|\nabla
^2u\|_{N-1}^{2}+\|\nabla n\|_{N-1}^{2}+\|\nabla^2\sigma\|_{N-1}^2
+\|\nabla ^{2}E\|_{N-3}^{2}+\|\nabla^{3}B\|_{N-4}^{2},
\end{equation}
respectively. Then, the global existence of  the reformulated Cauchy
problem \eqref{2.5}-\eqref{sec5.ggg} with small smooth initial data
can be stated as follows.

\begin{theorem}\label{th2}
Let $N\geq 4$. Assume that the compatibility condition \eqref{NC} holds true for  the initial data $U_{0}=[n_{0},u_{0},\sigma_{0},E_{0},B_{0}]$. If $\mathcal {E}_{N}(U_{0})>0$
is small enough,  the Cauchy problem
$\eqref{2.5}$-$\eqref{sec5.ggg}$ admits a unique global solution
$U=[n,u, \sigma,E,B]$ satisfying
\begin{eqnarray*}
\begin{aligned}
&[n(t,x),u(t,x),\sigma(t,x),E(t,x),B(t,x)]\in
C([0,\infty);H^{N}(\mathbb{R}^{3})),\\
&n\in L^{2}([0,\infty);H^{N }(\mathbb{R}^{3})),\\
&\nabla u \in L^{2}([0,\infty);H^{N}(\mathbb{R}^{3})),\ \ \ \nabla\sigma  \in L^{2}([0,\infty);H^{N}(\mathbb{R}^{3})),\\
&\nabla E\in L^{2}([0,\infty);H^{N-2}(\mathbb{R}^{3})),\ \ \ \nabla^{2} B\in L^{2}([0,\infty);H^{N-3}(\mathbb{R}^{3})),
\end{aligned}
\end{eqnarray*}
and
\begin{eqnarray}\label{energy.ineq}
\mathcal {E}_{N}(U(t))+\lambda\int_{0}^{t}\mathcal
{D}_{N}(U(s))ds\leq \mathcal {E}_{N}(U_{0}),
\end{eqnarray}
for any $t\geq 0$.
\end{theorem}

\begin{proof}We only  need to prove the following uniform-in-time a priori estimate.  For any smooth solution $U$ to the Cauchy problem
$\eqref{2.5}$-$\eqref{sec5.ggg}$ over $0\leq t \leq T$ with $T>0$, one has,
\begin{eqnarray}\label{PrioriE}
\frac{d}{dt}\mathcal {E}_{N}(U(t))+ c\mathcal
{D}_{N}(U(t)) \leq  \left(\mathcal {E}_{N}(U(t))+\mathcal {E}_{N}^{1/2}(U(t))\right)\mathcal
{D}_{N}(U(t)).
\end{eqnarray}
As long as the above estimate is proved, Theorem $\eqref{th2}$ follows in a standard way by combining it with the local-in-time existence and uniqueness as well as the continuity argument. Therefore, in what follows, we only prove \eqref{PrioriE} and other details are omitted for simplicity.  The proof can be divided into  four steps as follows.

\medskip
\noindent{\bf Step 1.}Denote  $\alpha_{1}=P_{\rho}(1,1),\
\alpha_{2}=P_{\theta}(1,1),\
\alpha_{3}=\frac{P_{\theta}(1,1)}{c_{\nu}},\
\bar{\kappa}=\frac{\kappa}{c_{\nu}}$. Zero-order estimate to the equations \eqref{2.5} imply that
\begin{equation}\label{3.3}
\begin{aligned}
&\frac{1}{2}\frac{d}{dt}\left(\alpha_{1}\|n\|_{N}^{2}+\|u\|_{N}^{2}+\frac{\alpha_{2}}{\alpha_{3}}\|\sigma\|_{N}^{2}+\|[E,B]\|_{N}^{2}\right)+\mu\|\nabla
u\|_{N}^{2}+(\mu+\mu')\|\nabla \cdot
u\|_{N}^{2}+\frac{\bar{\kappa}\alpha_{2}}{\alpha_{3}}\|\nabla
\sigma\|_{N}^{2}\\
=&\sum_{|\alpha|\leq
N}\left(\alpha_{1}\langle\partial^{\alpha}h_{1},\partial^{\alpha}
n\rangle+\langle\partial^{\alpha}h_{2},\partial^{\alpha}
u\rangle+\frac{\alpha_{2}}{\alpha_{3}}\langle\partial^{\alpha}h_{3},\partial^{\alpha}
\sigma\rangle+\langle\partial^{\alpha}h_{4},\partial^{\alpha}
E\rangle\right).
 \end{aligned}
\end{equation}

In what follows, we estimate the four terms on the  right-hand side. Recall that
\begin{equation*}
\arraycolsep=1.5pt \left\{
 \begin{aligned}
 &h_{1}=-\nabla\cdot(n u)=-u\cdot \nabla n-n\nabla\cdot u,\\
 &h_{2}\thicksim u\cdot \nabla u +(n+\sigma)\nabla n+(n+\sigma)\nabla \sigma+n\Delta u+n\nabla \nabla \cdot u+u\times B,\\
 &h_{3}\thicksim u\cdot \nabla \sigma +(n+\sigma)\nabla \cdot u+n\Delta \sigma +|\nabla u+ (\nabla u)'|^2+(\nabla \cdot u)^2.
\end{aligned}\right.
\end{equation*}
When $|\alpha|=0$, because of the similarity of these terms,  we only list some estimates for the clear reference.
\begin{equation*}
\arraycolsep=1.5pt
 \begin{aligned}
 &|\langle u\cdot \nabla n,n\rangle|\leq \|u\|_{L^{\infty}}\|\nabla n\|_{L^{2}}\| n\|_{L^{2}} \leq C\| n\|_{L^{2}} ( \|\nabla u\|_{H^{1}}^2+\|\nabla n\|^{2}_{L^{2}}),\\
 &\langle u\times B,u\rangle=0,\\
 & |\langle n\cdot \Delta u,u\rangle|\leq   \|n\|_{L^{\infty}}\|\Delta u\|_{L^{2}}\| u\|_{L^{2}} \leq C\| u\|_{L^{2}} ( \|\Delta u\|_{L^{2}}^2+\|\nabla n\|^{2}_{H^{1}}),\\
 &|\langle nu,E\rangle|\leq \|E\|\|u\|_{L^{\infty}}\|n\|\leq \|E\|_{L^{2}}\left(\|\nabla u\|_{H^{1}}^2+\|n\|^{2}\right).
\end{aligned}
\end{equation*}
 When $1\leq |\alpha|\leq N$ and set $|\alpha|=k$. We only estimate the following five terms, and other terms can be estimated in a similar way.

 For the terms containing $(N+1)$-order derivatives of $n$ or $(N+2)$-order derivatives of $u$, one can estimate them after integration by parts with respect to $x$.
\begin{equation*}
\arraycolsep=1.5pt
 \begin{aligned}
   \langle \partial^{\alpha}(u\cdot \nabla n),\partial^{\alpha}n\rangle
 &=\langle u\cdot \partial^{\alpha}\nabla n,\partial^{\alpha}n\rangle+\sum_{\beta<\alpha}C_{\alpha}^{\beta} \langle\partial^{\alpha-\beta}u\cdot \partial^{\beta}\nabla n,\partial^{\alpha}n\rangle\\
 &=-\frac{1}{2}\langle \nabla\cdot u,|\partial^{\alpha}n|^{2}\rangle+\sum_{\beta<\alpha}C_{\alpha}^{\beta} \langle\partial^{\alpha-\beta}u\cdot \partial^{\beta}\nabla n,\partial^{\alpha}n\rangle\\
 &\lesssim \|\nabla u\|_{N-1}\|\nabla n\|^{2}_{N-1}+\left(\|\nabla u\|_{L^{\infty}}\|\nabla^{k}n\|_{L^{2}} +\|\nabla n\|_{L^{\infty}}\|\nabla^{k}u\|_{L^{2}}\right)\|\nabla^{k}n\|_{L^{2}}\\
 &\lesssim \|\nabla u\|_{N-1}\|\nabla n\|^{2}_{N-1},
 \end{aligned}
\end{equation*}
 \begin{equation*}
\arraycolsep=1.5pt
 \begin{aligned}
   \langle \partial^{\alpha}(n \Delta u),\partial^{\alpha}u\rangle
 =&\langle n \partial^{\alpha}\Delta u,\partial^{\alpha}u\rangle+\sum_{\beta<\alpha}C_{\alpha}^{\beta} \langle\partial^{\alpha-\beta}n\cdot \partial^{\beta}\Delta u,\partial^{\alpha}u\rangle\\
= &-\langle \nabla n \partial^{\alpha}\nabla u,\partial^{\alpha}u\rangle-\langle n \partial^{\alpha}\nabla u,\partial^{\alpha}\nabla u\rangle +\sum_{\beta<\alpha}C_{\alpha}^{\beta} \langle\partial^{\alpha-\beta}n\cdot \partial^{\beta}\Delta u,\partial^{\alpha}u\rangle\\
 \lesssim & \|\nabla^2u\|_{N-1}\|\nabla u\|_{N-1}\|\nabla n\|_{N-1}+\|\nabla n\|_{N-1}\|\nabla^2u\|_{N-1}^2\\
 &+\left(\|\Delta u\|_{L^{\infty}}\|\nabla^{k}n\|_{L^{2}} +\|\nabla n\|_{L^{\infty}}\|\nabla^{k-1}\Delta u\|_{L^{2}}\right)\|\nabla^{k}u\|_{L^{2}}\\
\lesssim  &\|\nabla^2u\|_{N-1}\|\nabla u\|_{N-1}\|\nabla n\|_{N-1}+\|\nabla n\|_{N-1}\|\nabla^2u\|_{N-1}^2,
 \end{aligned}
\end{equation*}
and for the term containing $B$, one has
 \begin{equation*}
\arraycolsep=1.5pt
 \begin{aligned}
   \langle \partial^{\alpha}(u \times B),\partial^{\alpha}u\rangle
 =&\langle  \partial^{\alpha} u \times B,\partial^{\alpha}u\rangle+\sum_{\beta<\alpha}C_{\alpha}^{\beta} \langle\partial^{\beta} u\times \partial^{\alpha-\beta}B,\partial^{\alpha}u\rangle\\
= &\sum_{\beta<\alpha}C_{\alpha}^{\beta} \langle\partial^{\beta} u\times \partial^{\alpha-\beta}B,\partial^{\alpha}u\rangle\\
 \lesssim &\left(\|\nabla B\|_{L^{\infty}}\|\nabla^{k-1}u\|_{L^{2}} +\|u\|_{L^{\infty}}\|\nabla^{k}B\|_{L^{2}}\right)\|\nabla^{k}u\|_{L^{2}}\\
\lesssim  &\|u\|_{N-1}\|\nabla^2 B\|_{N-3}\|\nabla u\|_{N-1}+\|\nabla B\|_{N-1}\|\nabla u\|_{N-1}^2.
 \end{aligned}
\end{equation*}
For other terms, we can directly estimate them by using Lemma \ref{s5.le1} as follows.
\begin{equation*}
\arraycolsep=1.5pt
 \begin{aligned}
   \langle \partial^{\alpha}(u\cdot \nabla \sigma),\partial^{\alpha}\sigma\rangle
 &\lesssim  \left(\|  u\|_{L^{\infty}}\|\nabla^{k}\nabla \sigma\|_{L^{2}} +\|\nabla \sigma\|_{L^{\infty}}\|\nabla^{k}u\|_{L^{2}}\right)\|\nabla^{k}\sigma\|_{L^{2}}\\
 &\lesssim \|\nabla u\|_{N-1}\|\nabla^{2} \sigma\|_{N-1}\|\nabla \sigma\|_{N-1},
 \end{aligned}
\end{equation*}

\begin{equation*}
\arraycolsep=1.5pt
 \begin{aligned}
   \langle \partial^{\alpha}(nu),\partial^{\alpha}E\rangle
 \lesssim &\left(\|n\|_{L^{\infty}}\|\nabla^{k}u\|_{L^{2}} +\|u\|_{L^{\infty}}\|\nabla^{k}n\|_{L^{2}}\right)\|\nabla^{k}E\|_{L^{2}}\\
\lesssim  &\|\nabla u\|_{N-1}\|\nabla n\|_{N-1}\|\nabla E\|_{N-1}.
 \end{aligned}
\end{equation*}
By collecting the above estimates, one has
\begin{equation}\label{3.3*}
\begin{aligned}
&\frac{1}{2}\frac{d}{dt}\left(\alpha_{1}\|n\|_{N}^{2}+\|u\|_{N}^{2}+\frac{\alpha_{2}}{\alpha_{3}}\|\sigma\|_{N}^{2}+\|[E,B]\|_{N}^{2}\right)+\mu\|\nabla
u\|_{N}^{2}+(\mu+\mu')\|\nabla \cdot
u\|_{N}^{2}+\frac{\bar{\kappa}\alpha_{2}}{\alpha_{3}}\|\nabla
\sigma\|_{N}^{2}\\
\leq & C\mathcal {E}_{N}^{1/2}(U(t))\mathcal
{D}_{N}(U(t)).
 \end{aligned}
\end{equation}

\medskip
\noindent{\bf Step 2.}
 It holds that
\begin{eqnarray}\label{step2}
  \begin{aligned}
  &\frac{d}{dt}\sum_{|\alpha|\leq
N-1}\langle
\partial^{\alpha}u,
\partial^{\alpha}\nabla n\rangle
  +\lambda \|n\|_{N}^{2}\\
  \leq &
  C\|\nabla [u,\sigma]\|_{N}^{2}+\mathcal {E}_{N}(U(t))\mathcal
{D}_{N}(U(t)).
  \end{aligned}
\end{eqnarray}

In fact, recall the first two equations in \eqref{2.5},
\begin{equation}\label{step2equ}
\left\{
  \begin{aligned}
  &\partial_t n+\nabla\cdot u=h_{1},\\
  &\partial_t u+P_{\rho}(1,1)\nabla n+P_{\theta}(1,1)\nabla \sigma+E-\mu\Delta u-(\mu+\mu')\nabla\nabla \cdot u=h_{2},
  \end{aligned}
  \right.
  \end{equation}
with
\begin{equation*}
\arraycolsep=1.5pt \left\{
 \begin{aligned}
 &h_{1}=-\nabla\cdot(n u),\\
 &h_{2}\thicksim u\cdot \nabla u +(n+\sigma)\nabla n+(n+\sigma)\nabla \sigma+n\Delta u+n\nabla \nabla \cdot u+u\times B.
\end{aligned}\right.
\end{equation*}

Let $0\leq |\alpha|\leq N-1$, applying $\partial^{\alpha}$ to the
second equation of $(\ref{step2equ})$, multiplying it by
$\partial^{\alpha}\nabla n$, taking integrations in $x$, using
integration by parts and also the final equation of \eqref{2.5},
replacing $\partial_{t}n$ from $\eqref{step2equ}_{1}$, one has
\begin{equation*}
 \begin{aligned}
 &\frac{d}{dt}\langle
\partial^{\alpha}u,
\partial^{\alpha}\nabla n\rangle
+\alpha_{1}\|\partial^{\alpha}\nabla
n\|^{2}+\left\|\partial^{\alpha}n
\right\|^2\\
=&-\left\langle\alpha_{2}\partial^{\alpha}\nabla
\sigma,\partial^{\alpha}\nabla n\right\rangle
+\langle\mu \partial^{\alpha}\Delta u,\partial^{\alpha}\nabla n\rangle+\langle(\mu+\mu') \partial^{\alpha}\nabla \nabla \cdot u,\partial^{\alpha}\nabla n\rangle\\
&+\langle\partial^{\alpha}h_{2},\partial^{\alpha}\nabla n\rangle
+\|\partial^{\alpha}\nabla \cdot
u\|^{2}-\langle\partial^{\alpha}h_{1},\partial^{\alpha}\nabla \cdot
u\rangle.
 \end{aligned}
\end{equation*}
Then, it follows from the Cauchy-Schwarz inequality that
\begin{equation}\label{ineq.2}
 \begin{aligned}
 &\frac{d}{dt}\langle
\partial^{\alpha}u,
\partial^{\alpha}\nabla n\rangle
+\lambda(\|\partial^{\alpha}\nabla n\|^{2}+\left\|\partial^{\alpha}n
\right\|)^2\\
\leq& C(\|\nabla u\|_{N}^{2}+\|\nabla
\sigma\|_{N}^{2})+C\sum_{|\alpha|\leq
N-1}\|\partial^{\alpha}[h_{1},h_{2}]\|^{2}.
\end{aligned}
\end{equation}
Noticing that  $h_{1},\ h_{2}$ are quadratically
nonlinear,  by using the Leibniz formula and Lemma \ref{s5.le1} one has
\begin{equation*}
\|\partial^{\alpha}h_{1}\|^{2}+\|\partial^{\alpha}h_{2}\|^{2}\leq
C \mathcal {E}_{N}(U(t))\mathcal
{D}_{N}(U(t)).
\end{equation*}
Substituting this into (\ref{ineq.2}) and taking
the summation over $|\alpha|\leq N-1$ implies (\ref{step2}).

\medskip
\noindent{\bf Step 3.} It holds that
\begin{eqnarray}\label{step3}
\begin{aligned}
&\frac{d}{dt}\sum_{|\alpha| \leq N-2}\langle \partial^{\alpha}\nabla\times u,\partial^{\alpha}
\nabla\times E\rangle+\lambda\sum_{|\alpha| \leq N-2}\|\partial^{\alpha}\nabla\times E\|^{2}\\
\leq \frac{C}{\epsilon}&\sum_{|\alpha|\leq N}\|\partial^{\alpha}
\nabla u\|^{2}+\epsilon\sum_{1\leq |\alpha|\leq N-2}
\|\partial^{\alpha}\nabla\times B\|^{2}+C \mathcal {E}_{N}(U(t))\mathcal
{D}_{N}(U(t)).
\end{aligned}
\end{eqnarray}
Taking the curl of the second equation of \eqref{2.5}, one has
\begin{equation}
\partial_{t}\nabla\times u+\nabla\times E-
\mu\Delta\nabla\times u=\nabla\times h_{2}.
\end{equation}
Let $0\leq|\alpha|\leq N-2$, applying $\partial^{\alpha}$ to the above equation, multiplying the resultant equation by $\partial^{\alpha}\nabla\times E$, integrating it with respect to $x$ and replacing $\partial_{t} \partial^{\alpha} \nabla \times E$ from the fourth equation of \eqref{2.5}, one has
\begin{equation}\label{step31}
\begin{aligned}
&\frac{d}{dt}\langle \partial^{\alpha} \nabla \times u, \partial^{\alpha} \nabla \times E \rangle + \| \partial^{\alpha} \nabla \times E \|^{2}\\
=&\mu \langle \partial^{\alpha} \Delta \nabla \times u, \partial^{\alpha} \nabla \times E \rangle+\langle \partial^{\alpha} \nabla \times h_{2}, \partial^{\alpha} \nabla \times E \rangle+\langle \partial^{\alpha} \nabla \times u,  \partial^{\alpha} \nabla \times \nabla \times B \rangle\\
 &+\langle \partial^{\alpha} \nabla \times u,  \partial^{\alpha} \nabla \times u \rangle+\langle \partial^{\alpha} \nabla \times u,  \partial^{\alpha} \nabla \times h_{4} \rangle,
\end{aligned}
\end{equation}
which from Cauchy-Schwarz inequality further implies
\begin{eqnarray}\label{ste3ineq}
\begin{aligned}
&\frac{d}{dt}\langle \partial^{\alpha}\nabla\times u,\partial^{\alpha}
\nabla\times E\rangle+\lambda\|\partial^{\alpha}\nabla\times E\|^{2}\\
\leq \frac{C}{\epsilon}&\sum_{|\alpha|\leq N}\|\partial^{\alpha}
\nabla u\|^{2}+\epsilon\sum_{1\leq |\alpha|\leq N-2}
\|\partial^{\alpha}\nabla\times B\|^{2}+C\|\partial^{\alpha}\nabla\times h_{4}\|^{2}+C\|\partial^{\alpha}\nabla\times h_{2}\|^{2},
\end{aligned}
\end{eqnarray}
for an arbitrary constant $0<\epsilon\leq 1$. Here the third term in \eqref{step31} has been estimated on two cases. When $|\alpha|=0$,  it follows from Cauchy-Schwarz inequality that
$$|\langle \partial^{\alpha} \nabla \times u,  \partial^{\alpha} \nabla \times \nabla \times B \rangle |
\leq \frac{C}{\epsilon}\|\nabla u\|^{2}+\epsilon\sum_{|\alpha|=1}
\|\partial^{\alpha}\nabla\times B\|^{2}.$$
When $1\leq |\alpha|\leq N-2$, integrating by parts, one has
\begin{equation*}
\begin{aligned}
\langle \partial^{\alpha} \nabla \times u,  \partial^{\alpha} \nabla \times \nabla \times B \rangle
=&\langle \partial^{\alpha} \nabla\times\nabla \times u,  \partial^{\alpha} \nabla \times B \rangle\\
\leq & \frac{C}{\epsilon}\sum_{|\alpha|\leq N-1}\|\partial^{\alpha}
\nabla u\|^{2}+\epsilon\sum_{1\leq |\alpha|\leq N-2}
\|\partial^{\alpha}\nabla\times B\|^{2}.
\end{aligned}
\end{equation*}
One can estimate $\|\partial^{\alpha}\nabla\times h_{4}\|^{2}+\|\partial^{\alpha}\nabla\times h_{2}\|^{2}$ by using Leibniz formula and Lemma \ref{s5.le1} as follow,
\begin{equation*}
\|\partial^{\alpha}\nabla\times h_{4}\|^{2}+\|\partial^{\alpha}\nabla\times h_{2}\|^{2}\leq
C \mathcal {E}_{N}(U(t))\mathcal
{D}_{N}(U(t)).
\end{equation*}
Then, substituting this into (\ref{ste3ineq}) and taking
the summation over $|\alpha|\leq N-2$ implies (\ref{step3}).

\medskip
\noindent{\bf Step 4.} It holds that
\begin{eqnarray}\label{step4}
\begin{aligned}
&-\frac{d}{dt}\sum_{1\leq |\alpha| \leq N-2} \langle \partial^{\alpha} E, \partial^{\alpha} \nabla \times B\rangle + \lambda\sum_{1\leq |\alpha| \leq N-2}\| \partial^{\alpha} \nabla \times B \|^{2}\\
\leq &\sum_{1\leq |\alpha| \leq N-2}
\|\partial^{\alpha}\nabla\times E\|^{2}+
C\sum_{1\leq |\alpha| \leq N-2}
\|\partial^{\alpha}u\|^{2}+\|\nabla n\|_{N-1}^{2}\|\nabla u\|^{2}_{N-1}.\\
\end{aligned}
\end{eqnarray}
Let $1\leq |\alpha|\leq N-2$, applying $\partial^{\alpha}$ to the third equation in \eqref{2.5}, and take inner product with $-\partial^{\alpha} \nabla \times B$, one has,
\begin{equation*}
\begin{aligned}
&-\frac{d}{dt} \langle \partial^{\alpha} E, \partial^{\alpha} \nabla \times B\rangle + \| \partial^{\alpha} \nabla \times B \|_{L^{2}}^{2} \\
=&-\langle \partial^{\alpha} E, \partial_{t} \partial^{\alpha} \nabla \times B \rangle  - \langle \partial^{\alpha} h_{4}, \partial^{\alpha} \nabla \times  B \rangle- \langle \partial^{\alpha} u, \partial^{\alpha} \nabla \times B \rangle\\
=& \| \partial^{\alpha} \nabla \times E \|^{2}- \langle \partial^{\alpha} h_{4}, \partial^{\alpha} \nabla \times  B \rangle- \langle \partial^{\alpha} u, \partial^{\alpha} \nabla \times B \rangle,
\end{aligned}
\end{equation*}
where the fifth equation in \eqref{2.5} has been used to replace $\partial_{t} \partial^{\alpha} \nabla \times B$. It follows from Cauchy-Schwarz inequality that,
\begin{eqnarray}\label{}
\begin{aligned}
&-\frac{d}{dt} \langle \partial^{\alpha} E, \partial^{\alpha} \nabla \times B\rangle + \lambda\| \partial^{\alpha} \nabla \times B \|^{2}\\
\leq &\sum_{1\leq |\alpha| \leq N-2}
\|\partial^{\alpha}\nabla\times E\|^{2}+
C\sum_{1\leq |\alpha| \leq N-2}
\|\partial^{\alpha}u\|^{2}+\sum_{1\leq |\alpha| \leq N-2}\|\partial^{\alpha}h_{4}\|^{2}\\
\leq &\sum_{1\leq |\alpha| \leq N-2}
\|\partial^{\alpha}\nabla\times E\|^{2}+
C\sum_{1\leq |\alpha| \leq N-2}
\|\partial^{\alpha}u\|^{2}+\|\nabla n\|^{2}_{N-1}\|\nabla u\|^{2}_{N-1}.\\
\end{aligned}
\end{eqnarray}

Then, by putting the above estimates together for
properly chosen constants $0 \leq \kappa _{1},\kappa_{2},\epsilon \ll 1$ , one has

\begin{equation}\label{energy.final}
 \frac{d}{dt}\mathcal {E}_{N}(U(t))+c\mathcal {D}_{N}(U(t)) \leq \left(\mathcal {E}^{\frac{1}{2}}_{N}(U(t))+\mathcal {E}_{N}(U(t))\right)\mathcal {D}_{N}(U(t)).
\end{equation}
The proof of Theorem \ref{th2} is complete.

\end{proof}

\subsection{Large-time behavior}

In this section, we assume that all conditions in Theorem~\ref{th2} still hold and  $U=[\rho,u,\sigma,E,B]$ is the solution to the Cauchy problem $\eqref{2.5}$-$\eqref{NC}$. And now, we are going to study the time decay rate of the full energy $\|U(t)\|_{H^{N}}^{2}$. For this purpose, we define
$$X(t)=\sup_{0\leq s \leq t} (1+s)^{\frac{3}{4}}\mathcal {E}(U(t)),\ \ t \geq 0.$$
Then, we can derive the following Lemma.

\begin{lemma}\label{lemma.Xbounded}
If $\|U_{0}\|_{L^{1}\cap H^{N+1}}$ is sufficiently small, then
\begin{equation}\label{largetime-decay}
 \sup_{t\geq 0}X(t)\leq C \|U_{0}\|_{L^{1}\cap H^{N+1}}^{2}.
\end{equation}
\end{lemma}

\begin{proof}
Because of the smallness assumption of $\|U_{0}\|_{H^{N}}$, we know from the proof of Theorem~\ref{th2} that
\begin{equation}\label{ltb1}
\frac{d}{dt}\mathcal{E}_{N}(U(t))+c\mathcal{D}_{N}(U(t))
\leq 0
\end{equation}
for any $t\geq 0$. This is the basis by which we can deduce \eqref{largetime-decay}. First, fix a constant $\epsilon >0$ small enough. Then, a time weighted estimate on \eqref{ltb1} gives
\begin{eqnarray}\label{ltb2}
\begin{aligned}
(1+t)^{\frac{3}{4}+\epsilon}&\mathcal{E}(U(t))+c\int_{0}
^{t}(1+s)^{\frac{3}{4}+\epsilon}\mathcal{D}_{N}(U(t))ds\\
&\leq\mathcal{E}_{N}(U(0))+(\frac{3}{4}+\epsilon)\int_{0}^{t}
(1+s)^{-\frac{1}{4}+\epsilon}\mathcal{E}_{N}(U(s))ds.
\end{aligned}
\end{eqnarray}

Now, we are about to deal with the second term of the right-hand part in the above inequality.
Notice that
\begin{eqnarray*}
\begin{aligned}
\mathcal{E}_{N}(U)&\thicksim \|U\|_{H^{N}}^{2}\\
&\leq\mathcal{D}_{N+1}(U)+\|[u,\sigma,E,B]\|_{L^{2}}^{2}+
\| \nabla B\|_{L^{2}}^{2} \\
&\leq C \mathcal{D}_{N+1}(U)+C\|[u,\sigma,E,B]\|_{L^{2}}^{2},
\end{aligned}
\end{eqnarray*}
and
\begin{eqnarray*}
\int_{0}^{t}\mathcal{D}_{N+1}(U(s))ds\leq C\mathcal {E}_{N+1}(U_{0}),\\
\end{eqnarray*}
it follows that
\begin{eqnarray*}
\begin{aligned}
(1+t)^{\frac{3}{4}+\epsilon}&\mathcal{E}_{N}(U(t))+
\int_{0}^{t}(1+s)^{\frac{3}{4}+\epsilon}\mathcal{D}_{N}(U(s))ds\\
&\leq C\mathcal{E}_{N+1}(U_{0})+
C\int_{0}^{t}(1+s)^{-\frac{1}{4}+\epsilon}
\|[u,\sigma,E,B]\|_{L^{2}}^{2}ds.
\end{aligned}
\end{eqnarray*}
Here we have used the following inequality
$$
\| \nabla B\|_{L^{2}}^{2}\leq \| B\|_{L^{2}} \| \nabla^{2} B\|_{L^{2}} \leq \| B\|_{L^{2}}^{2} + \| \nabla^{2} B\|_{L^{2}}^{2}.
$$
Combining Lemma~\ref{lemma.decay3} with Duhamel's principle, one has
\begin{equation*}
\|\sigma\|\lesssim(1+t)^{-\frac{3}{4}}\|[n_{0},
u_{0},\sigma_{0}]\|_{L^{1}\cap L^{2}}+
\int_{0}^{t}(1+t-s)^{-\frac{3}{4}}\|[h_{1},
h_{2},h_{3}]\|_{L^{1}\cap L^{2}}ds.
\end{equation*}
Similarly, it is easy to obtain
\begin{multline*}
\|u\|\lesssim
(1+t)^{-\frac{3}{4}}\|[n_{0},u_{0},\sigma_{0},E_{0}]\|
_{L^{1}\cap L^{2}}+(1+t)^{-\frac{5}{8}}\|B_{0}\|_{L^{1}\cap L^{2}}\\
+\int_{0}^{t}(1+t-s)^{-\frac{3}{4}}\|[h_{1},
h_{2},h_{3},h_{4}]\|_{L^{1}\cap L^{2}}ds,
\end{multline*}
\begin{eqnarray*}
\begin{aligned}
\|E\|\lesssim(1+t)^{-\frac{3}{4}}\|U_{0}\|_{L^{1}\cap H^{3}}+
\int_{0}^{t}(1+t-s)^{-\frac{3}{4}}\|[h_{1},
h_{2},&h_{3},h_{4}]\|_{L^{1}\cap L^{2}}ds\\
&+\int_{0}^{t}(1+t)^{-\frac{3}{4}}\|\nabla^{2}h_{4}\|ds,
\end{aligned}
\end{eqnarray*}
and
\begin{eqnarray*}
  &&\begin{aligned}
\|B\|\lesssim(1+t)^{-\frac{3}{8}}\|[u_{0},E_{0},B_{0}]\|
_{L^{1}\cap H^{3}}+&\int_{0}^{t}(1+t-s)^{-\frac{5}{8}}
\|h_{2}\|_{L^{1}\cap L^{2}}ds\\
&+\int_{0}^{t}(1+t-s)^{-\frac{9}{8}}
(\|h_{4}\|_{L^{1}\cap L^{2}}+\|\nabla^{3}h_{4}\|)ds.
\end{aligned}
\end{eqnarray*}
It is straightforward to verify
\begin{equation*}
\|[h_{1},h_{2},h_{3},h_{4}]\|_{L^{1}\cap L^{2}}
+\|h_{4}\|_{ H^{3}}\leq C\mathcal{E}
_{N}(U).
\end{equation*}
Then, $\|B\|$ is estimated by
\begin{eqnarray*}
  &&\begin{aligned}
\|B\|&\lesssim(1+t)^{-\frac{3}{8}}\|[u_{0},E_{0},B_{0}]\|
_{L^{1}\cap H^{3}}+\int_{0}^{t}(1+t-s)^{-\frac{5}{8}}(1+s)
^{-\frac{3}{4}}ds X(t)\\
&\lesssim(1+t)^{-\frac{3}{8}}(\|[u_{0},E_{0},B_{0}]\|
_{L^{1}\cap H^{3}}+X(t)).
\end{aligned}
\end{eqnarray*}
In a similar way, we can find out that
\begin{equation*}
\|[u,E,B]\|\leq C(1+t)^{-\frac{3}{8}}(\|U_{0}\|
_{L^{1}\cap H^{3}}+X(t)).
\end{equation*}
Thus, one has
\begin{eqnarray}\label{ltb.final1}
\begin{aligned}
(1+t)^{\frac{3}{4}+\epsilon}&\mathcal{E}_{N}(U(t))+
\int_{0}^{t}(1+s)^{\frac{3}{4}+\epsilon}\mathcal{D}_{N}(U(t))ds\\
&\leq C\mathcal{E}_{N+1}(U_{0})+C\int_{0}^{t}
(1+s)^{-\frac{1}{4}+\epsilon}(1+s)^{-\frac{3}{4}}ds
(\|U_{0}\|^2
_{L^{1}\cap H^{3}}+X(t)^2)\\
&\leq C\mathcal{E}_{N+1}(U_{0})+C(1+t)^{\epsilon}
(\|U_{0}\|^{2}_{L^{1}\cap H^{3}}+X(t)^{2}),
\end{aligned}
\end{eqnarray}
which implies that
\begin{equation*}
X(t)\leq C(\|U_{0}\|^{2}_{L^{1}\cap H^{N+1}}+X(t)^{2}).
\end{equation*}
Since $\|U_{0}\|_{L^{1}\cap H^{N+1}}$ is sufficiently small, $X(t)$ is bounded uniformly in time and also \eqref{largetime-decay} holds true. This completes the proof of Lemma \ref{lemma.Xbounded}.

\end{proof}
\medskip

\subsection{Optimal large-time behavior}
In this section  we shall prove the main Theorem \ref{th1} on the
large-time asymptotic behavior of the obtained solutions.

\begin{theorem}\label{th3}
Let $N\geq 4$. We have the following optimal time-decay estimates,
\begin{eqnarray}\label{}
  &&\begin{aligned}
&\|n(t)\|\lesssim (1+t)^{-1}\|U_{0}\|_{L^{1}\cap H^{N+2}},\\
&\|u(t)\|\lesssim 1+t)^{-\frac{5}{8}}\|U_{0}\|_{L^{1}\cap H^{N+2}},\\
&\|\sigma(t)\|\leq C(1+t)^{-\frac{3}{4}}\ln(1+t)\|U_{0}\|_{L^{1}\cap H^{N+2}},\\
&\|E(t)\|\lesssim (1+t)^{-\frac{3}{4}}\ln(1+t)\|U_{0}\|_{L^{1}\cap H^{N+2}},\\
&\|B(t)\|\lesssim (1+t)^{-\frac{3}{8}}\|U_{0}\|_{L^{1}\cap H^{N+2}}.
\end{aligned}
\end{eqnarray}

\end{theorem}
\medskip

We shall prove the above theorem as follows. The estimate of $B$, in fact, has been obtained in Lemma~\ref{lemma.Xbounded}, the estimate of $u$ will be given in Lemma~\ref{lemma.Ybounded} and the estimate of $n$, $\sigma$ and $E$ will be given in Lemma~\ref{lemma.thelast}.
\begin{lemma}\label{lemma.Eh}
If $\|U_{0}\|_{H^{N}}$ is sufficiently small, then
\begin{equation}\label{oltb.lemma.Eh.0}
\frac{d}{dt}\mathcal{E}_{N}^{h}(U(t))+c
\mathcal{D}_{N}^{h}(U(t))\leq C(\|u\|^{2}+
\mathcal{E}_{N}^{h}(U(t)))\mathcal{E}_{N}^{h}(U(t))
\end{equation}
for any $t\geq 0$.
\end{lemma}

\begin{proof}
The proof can be done by modifying the proof of Theorem
$\ref{th2}$ a little. In fact, by making the energy estimates
on the only  high-order derivatives, then corresponding to
$\eqref{3.3*}$, it can be re-verified that
\begin{equation*}
\begin{aligned}
&\frac{1}{2}\frac{d}{dt}\left(\alpha_{1}\|\nabla n\|_{N-1}^{2}+\|\nabla u\|_{N-1}^{2}+\frac{\alpha_{2}}{\alpha_{3}}\|\nabla\sigma\|_{N-1}^{2}+\|\nabla [E,B]\|_{N-1}^{2}\right)+\mu\|\nabla^2
u\|_{N-1}^{2}+(\mu+\mu')\|\nabla\nabla \cdot
u\|_{N-1}^{2}\\
&+\frac{\bar{\kappa}\alpha_{2}}{\alpha_{3}}\|\nabla^2
\sigma\|_{N-1}^{2} \leq C\mathcal {E}^{h}_{N}(U(t))^{1/2}\mathcal
{D}^{h}_{N}(U(t))+C(\|u\|+\mathcal {E}^{h}_{N}(U(t))^{1/2})\mathcal {E}^{h}_{N}(U(t))^{1/2}\mathcal
{D}^{h}_{N}(U(t))^{1/2},
 \end{aligned}
\end{equation*}
where we modify the estimate for the term containing $B$ by two cases,
\begin{itemize}
 \item when $2\leq |\alpha|\leq N$,
 \begin{equation*}
\arraycolsep=1.5pt
 \begin{aligned}
   \langle \partial^{\alpha}(u \times B),\partial^{\alpha}u\rangle
 =&\langle  \partial^{\alpha} u \times B,\partial^{\alpha}u\rangle+\sum_{\beta<\alpha}C_{\alpha}^{\beta} \langle\partial^{\beta} u\times \partial^{\alpha-\beta}B,\partial^{\alpha}u\rangle\\
= &\sum_{\beta<\alpha}C_{\alpha}^{\beta} \langle\partial^{\beta} u\times \partial^{\alpha-\beta}B,\partial^{\alpha}u\rangle\\
 \lesssim &\left(\|\nabla B\|_{L^{\infty}}\|\nabla^{|\alpha|-1}u\|_{L^{2}} +\|u\|_{L^{\infty}}\|\nabla^{|\alpha|}B\|_{L^{2}}\right)\|\nabla^{|\alpha|}u\|_{L^{2}}\\
\lesssim  &\|\nabla u\|_{N-2}\|\nabla^2 B\|_{N-2}\|\nabla^2u\|_{N-2},
 \end{aligned}
\end{equation*}
\item when $|\alpha|=1$,
\begin{equation*}
\arraycolsep=1.5pt
 \begin{aligned}
   \langle \partial^{\alpha}(u \times B),\partial^{\alpha}u\rangle
 =&\langle  \partial^{\alpha} u \times B,\partial^{\alpha}u\rangle+\langle  u \times \partial^{\alpha} B,\partial^{\alpha}u\rangle \\
= &\langle  u \times \partial^{\alpha} B,\partial^{\alpha}u\rangle\\
 \lesssim & \|u\|_{L^{2}}\|\nabla B\|_{L^{2}} \|\nabla ^2u\|_{H^{1}}.
 \end{aligned}
\end{equation*}
\end{itemize}

Corresponding to $\eqref{step2}$, it can be re-verified that
\begin{eqnarray*}
   \frac{d}{dt}\sum_{1\leq |\alpha|\leq
N-1}\langle
\partial^{\alpha}u,
\partial^{\alpha}\nabla n\rangle
  +\lambda \|\nabla n\|_{N-1}^{2}
 \leq C\|\nabla^2[u,\sigma]\|_{N-1}^{2}+\mathcal {E}^{h}_{N}(U(t))\left(\mathcal
{D}^{h}_{N}(U(t))+\mathcal {E}^{h}_{N}(U(t))\right),
\end{eqnarray*}
where we have used the following estimate,
\begin{equation*}
\|\nabla[h_{1},h_{2}]\|_{H^{N-2}}^{2}\lesssim
\mathcal{E}_{N}^{h}(U(t))[\mathcal{E}_{N}^{h}(U(t))
+\mathcal{D}_{N}^{h}(U(t))].
\end{equation*}

Corresponding to $\eqref{step3}$, it can be re-verified that
\begin{eqnarray*}
\begin{aligned}
&\frac{d}{dt}\sum_{1\leq |\alpha| \leq N-2}\langle \partial^{\alpha}\nabla\times u,\partial^{\alpha}
\nabla\times E\rangle+\lambda\sum_{1\leq|\alpha| \leq N-2}\|\partial^{\alpha}\nabla\times E\|^{2}\\
\leq \frac{C}{\epsilon}&\sum_{1\leq|\alpha|\leq N}\|\partial^{\alpha}
\nabla u\|^{2}+\epsilon\sum_{2\leq |\alpha|\leq N-2}
\|\partial^{\alpha}\nabla\times B\|^{2}+C \mathcal {E}^{h}_{N}(U(t))\left(\mathcal
{D}^{h}_{N}(U(t))+\mathcal {E}^{h}_{N}(U(t))\right),
\end{aligned}
\end{eqnarray*}
where $\langle \partial^{\alpha} \nabla \times u,  \partial^{\alpha} \nabla \times \nabla \times B \rangle$ can be re-verified as follows.
When $|\alpha|=1$,  it follows from Cauchy-Schwarz inequality that
$$|\langle \partial^{\alpha} \nabla \times u,  \partial^{\alpha} \nabla \times \nabla \times B \rangle |
\leq \frac{C}{\epsilon}\|\nabla^2u\|^{2}+\epsilon\sum_{|\alpha|=2}
\|\partial^{\alpha}\nabla\times B\|^{2},$$
when $2\leq |\alpha|\leq N-2$, integrating by parts, one has
\begin{equation*}
\begin{aligned}
\langle \partial^{\alpha} \nabla \times u,  \partial^{\alpha} \nabla \times \nabla \times B \rangle
=&\langle \partial^{\alpha} \nabla\times\nabla \times u,  \partial^{\alpha} \nabla \times B \rangle\\
\leq & \frac{C}{\epsilon}\sum_{1\leq|\alpha|\leq N-1}\|\partial^{\alpha}
\nabla u\|^{2}+\epsilon\sum_{2\leq |\alpha|\leq N-2}
\|\partial^{\alpha}\nabla\times B\|^{2},
\end{aligned}
\end{equation*}
and when $1\leq|\alpha|\leq N-2$, the nonlinear terms can be re-estimated in the following.
\begin{equation*}
\|\partial^{\alpha}\nabla\times h_{4}\|^{2}+C\|\partial^{\alpha}\nabla\times h_{2}\|^{2}\leq
C \mathcal {E}^h_{N}(U(t))(\mathcal
{D}^h_{N}(U(t))+\mathcal {E}^h_{N}(U(t))).
\end{equation*}

Corresponding to $\eqref{step4}$, it can be re-verified that
\begin{eqnarray*}
\begin{aligned}
&-\frac{d}{dt}\sum_{2\leq |\alpha| \leq N-2} \langle \partial^{\alpha} E, \partial^{\alpha} \nabla \times B\rangle + \lambda\sum_{2\leq |\alpha| \leq N-2}\| \partial^{\alpha} \nabla \times B \|^{2}\\
\leq &\sum_{2\leq |\alpha| \leq N-2}
\|\partial^{\alpha}\nabla\times E\|^{2}+
C\sum_{2\leq |\alpha| \leq N-2}
\|\partial^{\alpha}u\|^{2}+\|\nabla n\|_{N-1}^{2}\|\nabla u\|^{2}_{N-1}.\\
\end{aligned}
\end{eqnarray*}

By using the above estimates, it follows that
\begin{eqnarray}\label{oltb.lemma.Eh.1}
\begin{aligned}
&\hspace*{3.5mm}\frac{d}{dt}\mathcal{E}_{N}^{h}(U(t))
+c\mathcal{D}_{N}^{h}(U(t))\\
&\begin{array}{ll}
\lesssim(\|u\|+\mathcal{E}_{N}^{h}(U(t))^{1/2})
\mathcal{E}_{N}^{h}(U(t))^{1/2}
\mathcal{D}_{N}^{h}(U(t))^{1/2}+\mathcal{E}_{N}^{h}(U(t))[\mathcal{E}_{N}^{h}
(U(t))+\mathcal{D}_{N}^{h}(U(t))].
\end{array}
\end{aligned}
\end{eqnarray}
Since $\|U_{0}\|_{H^{N}}$ is sufficiently small and so is $\mathcal{E}_{N}(U(t))$ uniformly for all $t\geq 0$ by \eqref{energy.ineq}, \eqref{oltb.lemma.Eh.0} holds by applying the Cauchy inequality to the first term in \eqref{oltb.lemma.Eh.1}. This completes the proof of Lemma~\ref{lemma.Eh}.

\end{proof}

Furthermore, Let us define
\begin{equation}
Y(t)=\sup_{0\leq s\leq t}{(1+s)^{\frac{5}{4}}\{
\|u(s)\|^{2}+\mathcal{E}_{N}^{h}(U(s))\}},\ \ t\geq 0.
\end{equation}
Similar to the derivation of the uniform-in-time bound of $X(t)$ in Lemma~\ref{lemma.Xbounded}, we have the following result to show the boundedness of $Y(t)$ for all $t\geq 0$ and thus the time decay rate of $\|u\|$.

\begin{lemma}\label{lemma.Ybounded}
If $\|U_{0}\|_{L^{1}\cap H^{N+2}}$ is sufficiently small, then
\begin{equation}
\sup_{t\geq 0}Y(t)\leq C\|U_{0}\|_{L^{1}\cap H^{N+2}}^{2}.
\end{equation}
\end{lemma}

\begin{proof}
For a further estimate for the solutions, we shall use the high-order energy inequality \eqref{oltb.lemma.Eh.0}. Fix $\epsilon>0$ small enough. The time-weighted estimate on \eqref{oltb.lemma.Eh.0} gives
\begin{eqnarray}\label{lemma.Ybounded.main}
\begin{aligned}
&(1+t)^{\frac{5}{4}+\epsilon}\mathcal{E}_{N}^{h}(U(t))
+\int_{0}^{t}(1+s)^{\frac{5}{4}+\epsilon}
\mathcal{D}_{N}^{h}(U(s))ds\\
&\hspace*{6mm}\begin{array}{ll}
\lesssim&\displaystyle\mathcal{E}_{N}(U_{0})+\int_{0}^{t}
(1+s)^{\frac{5}{4}+\epsilon}[\|u(s)\|^{2}+\mathcal{E}
_{N}^{h}(U(s))]\mathcal{E}
_{N}^{h}(U(s))ds\\
&\displaystyle+\int_{0}^{t}(1+s)^{\frac{1}{4}+\epsilon}
\mathcal{E}_{N}^{h}(U(s))ds.
\end{array}
\end{aligned}
\end{eqnarray}
The second term of the right-hand side is bounded by $Y(t)^{2}$ since
\begin{equation*}
\int_{0}^{t}(1+s)^{\frac{5}{4}+\epsilon}(1+s)^{-\frac
{5}{4}-\frac{5}{4}}ds\leq C\\
\end{equation*}
for $\epsilon>0$ small enough. Now we only need to deal with the last term on the right-hand side of \eqref{lemma.Ybounded.main}. It can be estimated as follows,
\begin{eqnarray*}
\begin{aligned}
&\int_{0}^{t}(1+s)^{\frac{1}{4}+\epsilon}\mathcal{E}
_{N}^{h}(U(s))ds\\
\leq & \int_{0}^{t}(1+s)^{\frac{1}{4}+
\epsilon}(\mathcal{D}_{N+1}(U(s))+\|\nabla B(s)\|
^{2})ds\\
\leq &\int_{0}^{t}(1+s)^{\frac{3}{4}+\epsilon}\mathcal{D}
_{N+1}(U(s))ds+\int_{0}^{t}(1+s)^{\frac{1}{4}+
\epsilon}\|\nabla B(s)\|
^{2}ds\\
\lesssim & (1+t)^{\epsilon}\|U_{0}\|
_{L^{1}\cap H^{N+2}}^{2}+\int_{0}^{t}(1+s)^
{\frac{1}{4}+\epsilon}\|\nabla B(s)\|^{2}ds,
\end{aligned}
\end{eqnarray*}
where  \eqref{ltb.final1} has been used. From Lemma~\ref{lemma.decay2}, we can estimate $\|\nabla B\|$ as
\begin{eqnarray*}
\begin{aligned}
\begin{array}{ll}
\|\nabla B\|\lesssim & (1+t)^{-\frac{7}{8}}(\|u_{0}\|_{L^{1}\cap L^{2}}+\|\nabla u_0\|_{L^{2}})+(1+t)^{-\frac{11}{8}}(\|E_{0}\displaystyle\|_{L^{1}\cap L^{2}}+\|\nabla^{4}E_{0}\|_{L^{2}})\\[2mm]
 &\displaystyle+(1+t)^{-\frac{5}{8}}(\|B_{0}\|_{L^{1}\cap L^{2}}+\|\nabla^{3}B_{0}\|)\\[2mm]
&\displaystyle+\int_{0}^{t}(1+t-s)^{-\frac{7}{8}}
(\|h_{2}(s)\|_{L^{1}\cap L^{2}}+\|\nabla h_{2}(s)\|)ds\\[2mm]
&\displaystyle+\int_{0}^{t}(1+t-s)^{-\frac{11}{8}}
(\|h_{3}(s)\|_{L^{1}\cap L^{2}}+\|\nabla^{4}h_{3}(s)\|)ds.
\end{array}
\end{aligned}
\end{eqnarray*}
Since
\begin{eqnarray*}
\begin{aligned}
\begin{array}{ll}
\|h_{2}(t)\|&\displaystyle_{L^{1}\cap L^{2}}+\|\nabla h_{2}(t)\|+\|h_{3}(t)\|_{L^{1}\cap L^{2}}+\|\nabla^{4}h_{3}(t)\|\\[3mm]
&\displaystyle\lesssim\|U_{0}\|_{L^{1}\cap H^{N+1}}(1+t)^{-\frac{3}{8}}
Y(t)^{1/2}(1+t)^{-\frac{5}{8}}\\[3mm]
&\displaystyle\lesssim\|U_{0}\|_{L^{1}\cap H^{N+1}}(1+t)^{-1}Y(t)^{1/2},
\end{array}
\end{aligned}
\end{eqnarray*}
it follows
\begin{eqnarray*}
&&\begin{aligned}
&\begin{array}{ll}
\|\nabla B\|\lesssim(1+t)^{-\frac{5}{8}}
&\displaystyle\|U_{0}\|_{L^{1}\cap H^{4}}\\[2mm]
&\displaystyle +\int_{0}^{t}(1+t-s)^{-\frac{7}{8}}
(\|h_{2}(s)\|_{L^{1}\cap L^{2}}+\|\nabla h_{2}(s)\|ds\\
&\displaystyle +\int_{0}^{t}(1+t-s)^{-\frac{11}{8}}
(\|h_{4}(s)\|_{L^{1}\cap L^{2}}+\|\nabla^{4} h_{4}(s)\|ds\\[2mm]
\end{array}\\
&\begin{array}{ll}
\hspace*{1cm}\lesssim(1+t)^{-\frac{5}{8}}&\displaystyle\|U_{0}\|
_{L^{1}\cap H^{4}}\\[2mm]
&\displaystyle +\int_{0}^{t}(1+t-s)^{-\frac{7}{8}}
\|U_{0}\|_{L^{1}\cap H^{N+1}}(1+s)^{-1}Y(t)^{1/2}ds\\
&\displaystyle +\int_{0}^{t}(1+t-s)^{-\frac{11}{8}}
\|U_{0}\|_{L^{1}\cap H^{N+1}}(1+s)^{-1}Y(t)^{1/2}ds\\[2mm]
\end{array}\\
&\hspace*{1cm}\lesssim(1+t)^{-\frac{5}{8}}\|U_{0}\|
_{L^{1}\cap H^{N+1}}(1+Y(t))^{1/2}.\\
\end{aligned}
\end{eqnarray*}

Combining the estimates above, one has
\begin{eqnarray*}
&&\begin{aligned}
(1+t)^{\frac{5}{4}+\epsilon}&\mathcal{E}_{N}^{h}(U(t))
+\int_{0}^{t}(1+s)^{\frac{5}{4}+\epsilon}\mathcal{D}
(U(t))ds\\[2mm]
&\begin{array}{ll}
\lesssim (1+t)^{\epsilon}\|U_{0}\|_{L^{1}\cap H^{N+2}}^{2}+&\displaystyle Y(t)^{2}\\[2mm]
&\displaystyle +\int_{0}^{t}(1+s)^{\frac{1}{4}+
\epsilon}(1+s)^{-\frac{5}{4}}\|U_{0}\|^2
_{L^{1}\cap H^{N+2}}(1+Y(s))ds\\
\end{array}\\[2mm]
&\lesssim (1+t)^{\epsilon}\|U_{0}\|_{L^{1}\cap H^{N+2}}^{2}(1+Y(s))+Y(t)^{2},
\end{aligned}
\end{eqnarray*}
which implies
\begin{equation}\label{oltb.lemma.Ybounded.1}
\sup_{0\leq s\leq t}(1+s)^{\frac{5}{4}}
\mathcal{E}_{N}^{h}(U(s))
\lesssim \|U_{0}\|_{L^{1}\cap H^{N+1}}^{2}
(1+Y(t))+Y(t)^{2}.
\end{equation}

To complete the proof, we still have to estimate the optimal time-decay of $\|u\|$.
From Lemma~\ref{lemma.decay3}, we have
\begin{eqnarray*}
&&\begin{aligned}
\|u\|&\lesssim (1+t)^{-\frac{5}{8}}\|U_{0}\|_{L^{1}
\cap L^{2}}+\int_{0}^{t}(1+t-s)^{-\frac{3}{4}}
\|[h_{1},h_{2},h_{3},h_{4}]\|_{L^{1}\cap L^{2}}ds\\[2mm]
&\lesssim (1+t)^{-\frac{5}{8}}\|U_{0}\|_{L^{1}
\cap L^{2}}+(1+t)^{-\frac{5}{8}}\|U_{0}\|_{L^{1}
\cap H^{N+1}}Y(t)^{1/2},
\end{aligned}
\end{eqnarray*}
that is,
\begin{equation}\label{oltb.lemma.Ybounded.2}
\sup_{0\leq s\leq t}(1+s)^{\frac{5}{4}}\|u(s)\|^{2}
\lesssim \|U_{0}\|_{L^{1}\cap H^{N+1}}^{2}Y(t)+\|U_{0}\|^2_{L^{1}
\cap L^{2}}.
\end{equation}
Combining \eqref{oltb.lemma.Ybounded.1} with \eqref{oltb.lemma.Ybounded.2}, one has
\begin{equation*}
Y(t)\leq C\|U_{0}\|_{L^{1}\cap H^{N+2}}^{2}(1+
Y(t))+Y(t)^{2}.\\
\end{equation*}
Therefore, we've completed the proof by some continuity argument since $\|U_{0}\|_{L^{1}\cap H^{N+2}}$ is sufficiently small.
\medskip
\end{proof}

The last problem we have to deal with is to get the time decay estimates of $n$, $\sigma$ and $E$. Actually, we have

\begin{lemma}\label{lemma.thelast}
If $\|U_{0}\|_{L^{1}\cap H^{N+2}}$ is sufficiently small, then
\begin{eqnarray}
&&\begin{aligned}
&\|n(t)\|\leq C\|U_{0}\|_{L^{1}\cap H^{N+2}}(1+t)^{-1},\\
&\|[\sigma(t),E(t)]\|\leq C\|U_{0}\|_{L^{1}\cap H^{N+2}}
(1+t)^{-\frac{3}{4}}\ln(1+t).
\end{aligned}
\end{eqnarray}
for any $t\geq 0$.
\end{lemma}

\begin{proof}
For $\|n(t)\|$,
\begin{equation*}
\|n(t)\|\lesssim(1+t)^{-\frac{5}{4}}\|U_{0}\|_{L^{1}\cap L^{2}}+\int_{0}^{t}(1+t-s)^{-\frac{5}{4}}\|[h_{1},h_{2},h_{3},h_{4}]\|_{L^{1}\cap L^{2}}ds\\
\end{equation*}
It is straightforward to verify that
\begin{equation*}
\begin{aligned}
\|[h_{1},h_{2},h_{3},h_{4}]\|_{L^{1}\cap L^{2}}
\lesssim &(1+t)^{-\frac{3}{8}} \|U_{0}\|_{L^{1}\cap H^{N+2}}(1+t)^{-\frac{5}{8}}Y(t)^{1/2}\\
\lesssim & (1+t)^{-1} \|U_{0}\|_{L^{1}\cap H^{N+2}}^{2}.
\end{aligned}
\end{equation*}
Thus,
\begin{equation*}
\|n(t)\|\leq C\|U_{0}\|_{L^{1}\cap H^{N+2}}(1+t)^{-1}.
\end{equation*}

As for $E(t)$, we know
\begin{eqnarray*}
\begin{aligned}
\|E(t)\|\lesssim (1+t)^{-\frac{3}{4}}\|U_{0}\|_{L^{1}\cap H^{3}}
+\int_{0}^{t}(1+t-s)^{-\frac{3}{4}}
&\|[h_{1},h_{2},h_{3},h_{4}]\|_{L^{1}\cap L^{2}}ds+\int_{0}^{t}(1+t-s)^{-\frac{3}{4}}\|\nabla^{4}
h_{4}\|ds.
\end{aligned}
\end{eqnarray*}
Obviously, one has
\begin{equation*}
\|[h_{1},h_{2},h_{3},h_{4}]\|_{L^{1}\cap L^{2}}
\lesssim (1+t)^{-1} \|U_{0}\|_{L^{1}\cap H^{N+2}}^{2},
\end{equation*}
and
\begin{equation*}
\|\nabla^{2}h_{4}\|\lesssim (1+t)^{-\frac{5}{4}} \|U_{0}\|_{L^{1}\cap H^{N+2}}^{2},
\end{equation*}
which implies that
\begin{equation*}
\int_{0}^{t}(1+t-s)^{-\frac{3}{4}}
\|[h_{1},h_{2},h_{3},h_{4}]\|_{L^{1}\cap L^{2}}ds
\leq (1+t)^{-\frac{3}{4}}\ln (1+t)\|U_{0}\|^2_
{L^{1}\cap H^{N+2}}.
\end{equation*}
It is easy to see
\begin{equation*}
\|E(t)\|\leq C\|U_{0}\|_{L^{1}\cap H^{N+2}}
(1+t)^{-\frac{3}{4}}\ln(1+t).
\end{equation*}

As for $\sigma(t)$, we know that
\begin{eqnarray*}
&&\begin{aligned}
\|\sigma(t)\|\lesssim (1+t)^{-\frac{3}{4}}\|[n_{0},u_{0},\sigma_{0}]\|_{L^{1}\cap L^{2}}
+\int_{0}^{t}(1+t-s)^{-\frac{3}{4}}
&\|[h_{1},h_{2},h_{3}]\|_{L^{1}\cap L^{2}}ds,
\end{aligned}
\end{eqnarray*}
Similar to the estimate on $E(t)$, one has
\begin{equation*}
\|\sigma(t)\|\leq C\|U_{0}\|_{L^{1}\cap H^{N+2}}
(1+t)^{-\frac{3}{4}}\ln(1+t).
\end{equation*}

This completes the proof.
\medskip
\end{proof}
\medskip
\noindent{\bf Acknowledgements:}\ \  The authors would like to thank Professor Changjiang Zhu for his continuous help. QQL was supported by the National Natural Science
Foundation of China (No. 11501217) and the Fundamental Research Funds for the Central Universities (No. 2015ZM188).

\end{document}